\documentclass[a4paper,DIVcalc,12pt]{amsart}
\usepackage[centering,total={15cm,24cm}]{geometry}
\usepackage[onehalfspacing]{setspace}

\usepackage[T1]{fontenc}
\usepackage[latin1]{inputenc}
\usepackage{lmodern}
\usepackage[english]{babel}

\usepackage{amsmath}
\usepackage{amssymb}
\usepackage{amsthm}
\usepackage{mathrsfs}

\usepackage{braket}
\usepackage{dsfont}
\usepackage{faktor}
\usepackage{rotating}
\usepackage[arrow, matrix, curve]{xy}
\usepackage{mathtools}
\usepackage{hhline}
\usepackage{colortbl}
\usepackage{bbm}
\usepackage{stmaryrd}
\usepackage{paralist}
\usepackage{setspace}
\usepackage{hyperref}

\usepackage{tikz}

\usetikzlibrary{arrows,patterns,decorations.pathmorphing}

\tikzset{>=latex}
\usepackage{tkz-euclide}
\usetikzlibrary{shapes.misc}
\tikzset{cross/.style={cross out, draw=black, fill=none, minimum size=2*(#1-\pgflinewidth), inner sep=0pt, outer sep=0pt}, cross/.default={2pt}}

\DeclareFontFamily{OT1}{pzc}{}
\DeclareFontShape{OT1}{pzc}{m}{it}{<-> s * [1.10] pzcmi7t}{}
\DeclareMathAlphabet{\mathpzc}{OT1}{pzc}{m}{it}

\DeclareFontFamily{U}{mathx}{\hyphenchar\font45}
\DeclareFontShape{U}{mathx}{m}{n}{<-> mathx10}{}
\DeclareSymbolFont{mathx}{U}{mathx}{m}{n}
\DeclareMathAccent{\widebar}{0}{mathx}{"73}

\vfuzz \hfuzz
\newtheorem{prop}{Proposition}
\newtheorem{thm}{Theorem}

\theoremstyle{definition}
\newtheorem{defi}{Definition}

\newtheorem{expl}{Example}

\theoremstyle{remark}

\newcommand{\iso}{\xrightarrow{\raisebox{-0.7ex}[0ex][0ex]{$\sim$}}}

\begin{document}

\title{Counting (tropical) curves via scattering.sage}

\author{Tim Graefnitz} 
\address{\tiny Tim Gr\"afnitz, University of Cambridge, DPMMS, Wilberforce Road, Cambridge, CB3 0WB, UK}
\email{tim.graefnitz@gmx.de}

\begin{abstract}
In this note I will explain how relative/log Gromov-Witten invariants of pairs $(X,D)$ with very ample smooth anticanonical divisor $D$ can be computed using algebro-combinatorial objects called scattering diagrams. The underlying principle behind this computational method is a tropical correspondence theorem for non-toric cases, which I will explain briefly. This is based on \cite{Gra}. By computing some examples I will give an introduction to a sage code that I wrote for computing scattering diagrams.
\end{abstract}

\maketitle
\thispagestyle{empty}

Let $X$ be a smooth projective surface with very ample anticanonical bundle $-K_X$, i.e., a del Pezzo surface of degree $K_X^2\geq 3$, and let $D\in -K_X$ be a smooth divisor. We call $(X,D)$ a \textit{very ample log Calabi-Yau pair}, since $K_X+D$ is trivial. Note that $X=\mathbb{P}^1\times\mathbb{P}^1$ (degree $8$) or $X$ is the blow up of $\mathbb{P}^2$ in $n\leq 6$ points (degree $9-n$).

We are interested in counting rational curves on $X$ that intersect $D$ in a single (unspecified) point with maximal tangency. Formally these can be defined as relative or, equivalently, logarithmic Gromov-Witten invariants $R_\beta(X,D)$ of $(X,D)$ of a given class $\beta$ and genus $0$ \cite{GSloggw}. The invariants $R_\beta(X,D)$ do not only count irreducible curves. There are contributions coming from reducible curves and multiple covers of curves of lower degree.

\begin{figure}[h!]
\centering
\begin{tikzpicture}[scale=1.9]
\draw[smooth,samples=100,domain=-0.5:0.75] plot (\x,{sqrt((2*\x)^3+1)/2+4.75});
\draw[smooth,samples=100,domain=-0.5:0.75] plot (\x,{-sqrt((2*\x)^3+1)/2+4.75});
\draw[blue] (-0.75,5.25) -- (1,5.25);
\draw[black!40!green] (-0.68,4.75) ellipse (5pt and 10pt);
\coordinate (4) at (0,4.25);
\coordinate (5) at (0.5,4.5);
\end{tikzpicture}
\caption{A line (blue) and an irreducible conic (green) in $\mathbb{P}^2$ maximally tangent to an elliptic curve $E$ at a single point.}
\label{fig:classical}
\end{figure}
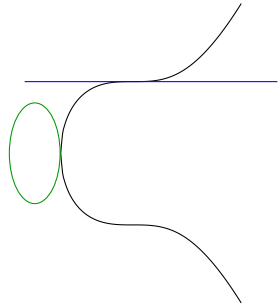

\begin{expl}
\label{expl:R}
The easiest example is $(\mathbb{P}^2,E)$, where $E\subset\mathbb{P}^2$ is an elliptic curve. Then $R_1(\mathbb{P}^2,E)=9$, since there are nine lines intersecting $E$ in exactly one point, one of the nine flex points of $E$. Further, $R_2(\mathbb{P}^2,E)=\frac{135}{4}=27+9\cdot\frac{3}{4}$, where $27$ is the number of irreducible conics maximally tangent to $E$, and each of the $9$ lines contributes $\frac{3}{4}$ by \cite{GPS}, Proposition 6.1. The next number is $R_3(\mathbb{P}^2,E)=244=234+9\cdot\frac{10}{9}$.
\end{expl}

\vspace{-3mm}
\rule{0.3\textwidth}{0.4pt} \\[0mm]
This work was financially supported by the ERC Advanced Grant MSAG.

\section{The scattering diagram for $(X,D)$}

\subsection{Toric models}

Let $(X,D)$ be a very ample log Calabi-Yau pair. Let $X^{tor}$ be a \textit{toric model} of $X$, i.e., a Gorenstein toric variety that admits a $\mathbb{Q}$-Gorenstein deformation to $X$. This means there exists a family $\mathcal{X}\rightarrow T$ with special fiber $X^{tor}$ and general fiber $X$ such that the total space $\mathcal{X}$ is $\mathbb{Q}$-Gorenstein. If $X$ is already toric we can take $X^{tor}=X$, but there might be other toric models. Let $\Sigma$ be the fan of $X^{tor}$ and let $\Delta$ be the \textit{spanning polytope} of $\Sigma$, i.e., the convex hull of the ray generators of $\Sigma$. Since $X^{tor}$ is Gorenstein and Fano this is a reflexive polytope. In dimension $2$ this is equivalent to the condition that $\Delta$ contains exactly one interior lattice point, and there are exactly $16$ such polytopes, see Figure \ref{fig:16}. The labeling in Figure \ref{fig:16} is such that the number gives the degree $K_X^2=K_{X^{tor}}^2$ and there is a letter iff $X\neq X^{tor}$. The cases (8') and (8'a) correspond to $\mathbb{P}^1\times\mathbb{P}^1$, while all other cases correspond to a blow up of $\mathbb{P}^2$ in $n\leq 6$ points.

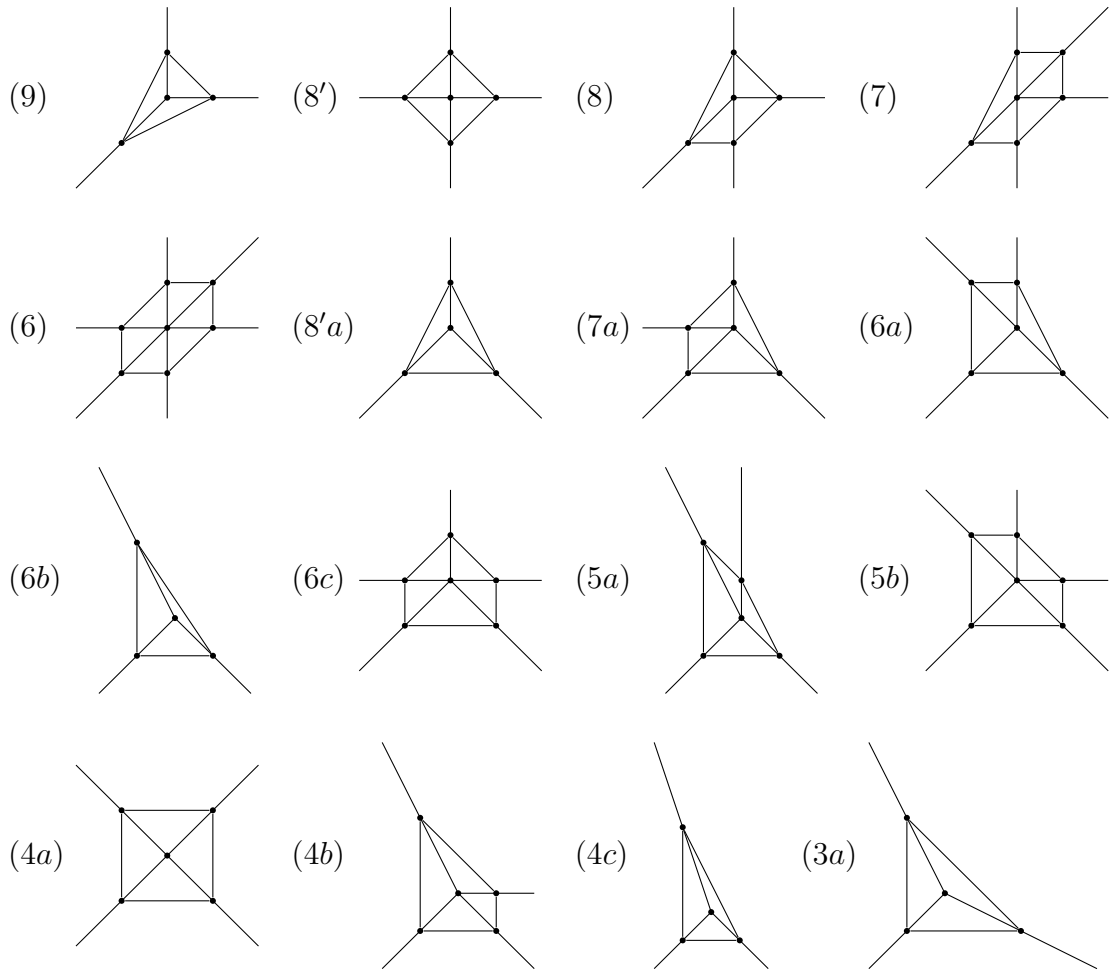
\begin{figure}[h!]
\centering
\begin{minipage}[c]{0.05\textwidth}
$(9)$
\end{minipage}
\begin{minipage}[c]{0.18\textwidth}
\begin{tikzpicture}[scale=0.6]
\coordinate[fill,circle,inner sep=0.8pt] (0) at (0,0);
\coordinate[fill,circle,inner sep=0.8pt] (1) at (1,0);
\coordinate[fill,circle,inner sep=0.8pt] (2) at (0,1);
\coordinate[fill,circle,inner sep=0.8pt] (3) at (-1,-1);
\coordinate (1a) at (2,0);
\coordinate (2a) at (0,2);
\coordinate (3a) at (-2,-2);
\coordinate (12) at (0.5,0.5);
\coordinate (12a) at (0.5,2);
\coordinate (12b) at (2,0.5);
\coordinate (23) at (-0.5,0);
\coordinate (23a) at (-0.5,2);
\coordinate (23b) at (-2,-1.5);
\coordinate (31) at (0,-0.5);
\coordinate (31a) at (2,-0.5);
\coordinate (31b) at (-1.5,-2);
\draw (1) -- (2) -- (3) -- (1);
\draw (0,0) -- (1a);
\draw (0,0) -- (2a);
\draw (0,0) -- (3a);
\end{tikzpicture}
\end{minipage}
\begin{minipage}[c]{0.05\textwidth}
$(8')$
\end{minipage}
\begin{minipage}[c]{0.18\textwidth}
\begin{tikzpicture}[scale=0.6]
\coordinate[fill,circle,inner sep=0.8pt] (0) at (0,0);
\coordinate[fill,circle,inner sep=0.8pt] (1) at (1,0);
\coordinate[fill,circle,inner sep=0.8pt] (2) at (0,1);
\coordinate[fill,circle,inner sep=0.8pt] (3) at (-1,0);
\coordinate[fill,circle,inner sep=0.8pt] (4) at (0,-1);
\coordinate (1a) at (2,0);
\coordinate (2a) at (0,2);
\coordinate (3a) at (-2,0);
\coordinate (4a) at (0,-2);
\coordinate (12) at (0.5,0.5);
\coordinate (12a) at (0.5,2);
\coordinate (12b) at (2,0.5);
\coordinate (23) at (-0.5,0.5);
\coordinate (23a) at (-0.5,2);
\coordinate (23b) at (-2,0.5);
\coordinate (34) at (-0.5,-0.5);
\coordinate (34a) at (-2,-0.5);
\coordinate (34b) at (-0.5,-2);
\coordinate (41) at (0.5,-0.5);
\coordinate (41a) at (2,-0.5);
\coordinate (41b) at (0.5,-2);
\draw (1) -- (2) -- (3) -- (4) -- (1);
\draw (0,0) -- (1a);
\draw (0,0) -- (2a);
\draw (0,0) -- (3a);
\draw (0,0) -- (4a);
\end{tikzpicture}
\end{minipage}
\begin{minipage}[c]{0.05\textwidth}
$(8)$
\end{minipage}
\begin{minipage}[c]{0.18\textwidth}
\begin{tikzpicture}[scale=0.6]
\coordinate[fill,circle,inner sep=0.8pt] (0) at (0,0);
\coordinate[fill,circle,inner sep=0.8pt] (1) at (-1,-1);
\coordinate[fill,circle,inner sep=0.8pt] (2) at (0,-1);
\coordinate[fill,circle,inner sep=0.8pt] (3) at (1,0);
\coordinate[fill,circle,inner sep=0.8pt] (4) at (0,1);
\coordinate (1a) at (-2,-2);
\coordinate (2a) at (0,-2);
\coordinate (3a) at (2,0);
\coordinate (4a) at (0,2);
\coordinate (12) at (-0.5,-1);
\coordinate (12a) at (-0.5,-2);
\coordinate (12b) at (-1.5,-2);
\coordinate (23) at (0.5,-0.5);
\coordinate (23a) at (0.5,-2);
\coordinate (23b) at (2,-0.5);
\coordinate (34) at (0.5,0.5);
\coordinate (34a) at (2,0.5);
\coordinate (34b) at (0.5,2);
\coordinate (41) at (-0.5,-0);
\coordinate (41a) at (-2,-1.5);
\coordinate (41b) at (-0.5,2);
\draw (1) -- (2) -- (3) -- (4) -- (1);
\draw (0,0) -- (1a);
\draw (0,0) -- (2a);
\draw (0,0) -- (3a);
\draw (0,0) -- (4a);
\end{tikzpicture}
\end{minipage}
\begin{minipage}[c]{0.05\textwidth}
$(7)$
\end{minipage}
\begin{minipage}[c]{0.18\textwidth}
\begin{tikzpicture}[scale=0.6]
\coordinate[fill,circle,inner sep=0.8pt] (0) at (0,0);
\coordinate[fill,circle,inner sep=0.8pt] (1) at (-1,-1);
\coordinate[fill,circle,inner sep=0.8pt] (2) at (0,-1);
\coordinate[fill,circle,inner sep=0.8pt] (3) at (1,0);
\coordinate[fill,circle,inner sep=0.8pt] (4) at (1,1);
\coordinate[fill,circle,inner sep=0.8pt] (5) at (0,1);
\coordinate (1a) at (-2,-2);
\coordinate (2a) at (0,-2);
\coordinate (3a) at (2,0);
\coordinate (4a) at (2,2);
\coordinate (5a) at (0,2);
\coordinate (12) at (-0.5,-1);
\coordinate (12a) at (-0.5,-2);
\coordinate (12b) at (-1.5,-2);
\coordinate (23) at (0.5,-0.5);
\coordinate (23a) at (0.5,-2);
\coordinate (23b) at (2,-0.5);
\coordinate (34) at (1,0.5);
\coordinate (34a) at (2,0.5);
\coordinate (34b) at (2,1.5);
\coordinate (45) at (0.5,1);
\coordinate (45a) at (1.5,2);
\coordinate (45b) at (0.5,2);
\coordinate (51) at (-0.5,0);
\coordinate (51a) at (-2,-1.5);
\coordinate (51b) at (-0.5,2);
\draw (1) -- (2) -- (3) -- (4) -- (5) -- (1);
\draw (0,0) -- (1a);
\draw (0,0) -- (2a);
\draw (0,0) -- (3a);
\draw (0,0) -- (4a);
\draw (0,0) -- (5a);
\end{tikzpicture}
\end{minipage}\\[6mm]
\begin{minipage}[c]{0.05\textwidth}
$(6)$
\end{minipage}
\begin{minipage}[c]{0.18\textwidth}
\begin{tikzpicture}[scale=0.6]
\coordinate[fill,circle,inner sep=0.8pt] (0) at (0,0);
\coordinate[fill,circle,inner sep=0.8pt] (1) at (1,1);
\coordinate[fill,circle,inner sep=0.8pt] (2) at (0,1);
\coordinate[fill,circle,inner sep=0.8pt] (3) at (-1,0);
\coordinate[fill,circle,inner sep=0.8pt] (4) at (-1,-1);
\coordinate[fill,circle,inner sep=0.8pt] (5) at (0,-1);
\coordinate[fill,circle,inner sep=0.8pt] (6) at (1,0);
\coordinate (1a) at (2,2);
\coordinate (2a) at (0,2);
\coordinate (3a) at (-2,0);
\coordinate (4a) at (-2,-2);
\coordinate (5a) at (0,-2);
\coordinate (6a) at (2,0);
\coordinate (12) at (0.5,1);
\coordinate (12a) at (0.5,2);
\coordinate (12b) at (1.5,2);
\coordinate (23) at (-0.5,0.5);
\coordinate (23a) at (-0.5,2);
\coordinate (23b) at (-2,0.5);
\coordinate (34) at (-1,-0.5);
\coordinate (34a) at (-2,-0.5);
\coordinate (34b) at (-2,-1.5);
\coordinate (45) at (-0.5,-1);
\coordinate (45a) at (-1.5,-2);
\coordinate (45b) at (-0.5,-2);
\coordinate (56) at (0.5,-0.5);
\coordinate (56a) at (2,-0.5);
\coordinate (56b) at (0.5,-2);
\coordinate (61) at (1,0.5);
\coordinate (61a) at (2,0.5);
\coordinate (61b) at (2,1.5);
\draw (1) -- (2) -- (3) -- (4) -- (5) -- (6) -- (1);
\draw (0,0) -- (1a);
\draw (0,0) -- (2a);
\draw (0,0) -- (3a);
\draw (0,0) -- (4a);
\draw (0,0) -- (5a);
\draw (0,0) -- (6a);
\end{tikzpicture}
\end{minipage}
\begin{minipage}[c]{0.05\textwidth}
$(8'a)$
\end{minipage}
\begin{minipage}[c]{0.18\textwidth}
\begin{tikzpicture}[scale=0.6]
\coordinate[fill,circle,inner sep=0.8pt] (0) at (0,0);
\coordinate[fill,circle,inner sep=0.8pt] (1) at (-1,-1);
\coordinate[fill,circle,inner sep=0.8pt] (2) at (1,-1);
\coordinate[fill,circle,inner sep=0.8pt] (3) at (0,1);
\coordinate (1a) at (-2,-2);
\coordinate (2a) at (2,-2);
\coordinate (3a) at (0,2);
\coordinate (12) at (-0.5,-1);
\coordinate (12a) at (-1.5,-2);
\coordinate (12b) at (-0.5,-2);
\coordinate (12') at (0.5,-1);
\coordinate (12'a) at (0.5,-2);
\coordinate (12'b) at (1.5,-2);
\coordinate (23) at (0.5,0);
\coordinate (23a) at (0.5,2);
\coordinate (23b) at (2,-1.5);
\coordinate (31) at (-0.5,0);
\coordinate (31a) at (-0.5,2);
\coordinate (31b) at (-2,-1.5);
\draw (1) -- (2) -- (3) -- (1);
\draw (0,0) -- (1a);
\draw (0,0) -- (2a);
\draw (0,0) -- (3a);
\end{tikzpicture}
\end{minipage}
\begin{minipage}[c]{0.05\textwidth}
$(7a)$
\end{minipage}
\begin{minipage}[c]{0.18\textwidth}
\begin{tikzpicture}[scale=0.6]
\coordinate[fill,circle,inner sep=0.8pt] (0) at (0,0);
\coordinate[fill,circle,inner sep=0.8pt] (1) at (1,-1);
\coordinate[fill,circle,inner sep=0.8pt] (2) at (-1,-1);
\coordinate[fill,circle,inner sep=0.8pt] (3) at (-1,0);
\coordinate[fill,circle,inner sep=0.8pt] (4) at (0,1);
\coordinate (1a) at (2,-2);
\coordinate (2a) at (-2,-2);
\coordinate (3a) at (-2,0);
\coordinate (4a) at (0,2);
\coordinate (12) at (0.5,-1);
\coordinate (12a) at (1.5,-2);
\coordinate (12b) at (0.5,-2);
\coordinate (12') at (-0.5,-1);
\coordinate (12'a) at (-0.5,-2);
\coordinate (12'b) at (-1.5,-2);
\coordinate (23) at (-1,-0.5);
\coordinate (23a) at (-2,-1.5);
\coordinate (23b) at (-2,-0.5);
\coordinate (34) at (-0.5,0.5);
\coordinate (34a) at (-2,0.5);
\coordinate (34b) at (-0.5,2);
\coordinate (41) at (0.5,0);
\coordinate (41a) at (2,-1.5);
\coordinate (41b) at (0.5,2);
\draw (1) -- (2) -- (3) -- (4) -- (1);
\draw (0,0) -- (1a);
\draw (0,0) -- (2a);
\draw (0,0) -- (3a);
\draw (0,0) -- (4a);
\end{tikzpicture}
\end{minipage}
\begin{minipage}[c]{0.05\textwidth}
$(6a)$
\end{minipage}
\begin{minipage}[c]{0.18\textwidth}
\begin{tikzpicture}[scale=0.6]
\coordinate[fill,circle,inner sep=0.8pt] (0) at (0,0);
\coordinate[fill,circle,inner sep=0.8pt] (1) at (1,-1);
\coordinate[fill,circle,inner sep=0.8pt] (2) at (-1,-1);
\coordinate[fill,circle,inner sep=0.8pt] (3) at (-1,1);
\coordinate[fill,circle,inner sep=0.8pt] (4) at (0,1);
\coordinate (1a) at (2,-2);
\coordinate (2a) at (-2,-2);
\coordinate (3a) at (-2,2);
\coordinate (4a) at (0,2);
\coordinate (12a) at (1.5,-2);
\coordinate (12b) at (0.5,-2);
\coordinate (12') at (-0.5,-1);
\coordinate (12'a) at (-0.5,-2);
\coordinate (12'b) at (-1.5,-2);
\coordinate (23) at (-1,0.5);
\coordinate (23a) at (-2,1.5);
\coordinate (23b) at (-2,0.5);
\coordinate (23') at (-1,-0.5);
\coordinate (23'a) at (-2,-0.5);
\coordinate (23'b) at (-2,-1.5);
\coordinate (34) at (-0.5,1);
\coordinate (34a) at (-0.5,2);
\coordinate (34b) at (-1.5,2);
\coordinate (41) at (0.5,0);
\coordinate (41a) at (0.5,2);
\coordinate (41b) at (2,-1.5);
\draw (1) -- (2) -- (3) -- (4) -- (1);
\draw (0,0) -- (1a);
\draw (0,0) -- (2a);
\draw (0,0) -- (3a);
\draw (0,0) -- (4a);
\end{tikzpicture}
\end{minipage}\\[6mm]
\begin{minipage}[c]{0.07\textwidth}
$(6b)$
\end{minipage}
\begin{minipage}[c]{0.16\textwidth}
\begin{tikzpicture}[scale=0.5]
\coordinate[fill,circle,inner sep=0.8pt] (0) at (0,0);
\coordinate[fill,circle,inner sep=0.8pt] (1) at (-1,-1);
\coordinate[fill,circle,inner sep=0.8pt] (2) at (1,-1);
\coordinate[fill,circle,inner sep=0.8pt] (3) at (-1,2);
\coordinate (1a) at (-2,-2);
\coordinate (2a) at (2,-2);
\coordinate (3a) at (-2,4);
\coordinate (12) at (-0.5,-1);
\coordinate (12a) at (-1.5,-2);
\coordinate (12b) at (-0.5,-2);
\coordinate (12') at (0.5,-1);
\coordinate (12'a) at (0.5,-2);
\coordinate (12'b) at (1.5,-2);
\coordinate (23) at (0,0.5);
\coordinate (23a) at (2,-1.5);
\coordinate (23b) at (-1.75,4);
\coordinate (31) at (-1,-0.5);
\coordinate (31a) at (-2,-1.5);
\coordinate (31b) at (-2,-0.5);
\coordinate (31') at (-1,0.5);
\coordinate (31'a) at (-2,0.5);
\coordinate (31'b) at (-2,1.5);
\coordinate (31'') at (-1,1.5);
\coordinate (31''a) at (-2,2.5);
\coordinate (31''b) at (-2,3.5);
\draw (1) -- (2) -- (3) -- (1);
\draw (0,0) -- (1a);
\draw (0,0) -- (2a);
\draw (0,0) -- (3a);
\end{tikzpicture}
\end{minipage}
\begin{minipage}[c]{0.05\textwidth}
$(6c)$
\end{minipage}
\begin{minipage}[c]{0.18\textwidth}
\begin{tikzpicture}[scale=0.6]
\coordinate[fill,circle,inner sep=0.8pt] (0) at (0,0);
\coordinate[fill,circle,inner sep=0.8pt] (1) at (-1,-1);
\coordinate[fill,circle,inner sep=0.8pt] (2) at (1,-1);
\coordinate[fill,circle,inner sep=0.8pt] (3) at (1,0);
\coordinate[fill,circle,inner sep=0.8pt] (4) at (0,1);
\coordinate[fill,circle,inner sep=0.8pt] (5) at (-1,0);
\coordinate (1a) at (-2,-2);
\coordinate (2a) at (2,-2);
\coordinate (3a) at (2,0);
\coordinate (4a) at (0,2);
\coordinate (5a) at (-2,0);
\coordinate (12) at (-0.5,-1);
\coordinate (12a) at (-0.5,-2);
\coordinate (12b) at (-1.5,-2);
\coordinate (12') at (0.5,-1);
\coordinate (12'a) at (0.5,-2);
\coordinate (12'b) at (1.5,-2);
\coordinate (23) at (1,-0.5);
\coordinate (23a) at (2,-1.5);
\coordinate (23b) at (2,-0.5);
\coordinate (34) at (0.5,0.5);
\coordinate (34a) at (2,0.5);
\coordinate (34b) at (0.5,2);
\coordinate (45) at (-0.5,0.5);
\coordinate (45a) at (-0.5,2);
\coordinate (45b) at (-2,0.5);
\coordinate (51) at (-1,-0.5);
\coordinate (51a) at (-2,-1.5);
\coordinate (51b) at (-2,-0.5);
\draw (1) -- (2) -- (3) -- (4) -- (5) -- (1);
\draw (0,0) -- (1a);
\draw (0,0) -- (2a);
\draw (0,0) -- (3a);
\draw (0,0) -- (4a);
\draw (0,0) -- (5a);
\end{tikzpicture}
\end{minipage}
\begin{minipage}[c]{0.07\textwidth}
$(5a)$
\end{minipage}
\begin{minipage}[c]{0.16\textwidth}
\begin{tikzpicture}[scale=0.5]
\coordinate[fill,circle,inner sep=0.8pt] (0) at (0,0);
\coordinate[fill,circle,inner sep=0.8pt] (1) at (1,-1);
\coordinate[fill,circle,inner sep=0.8pt] (2) at (-1,-1);
\coordinate[fill,circle,inner sep=0.8pt] (3) at (-1,2);
\coordinate[fill,circle,inner sep=0.8pt] (4) at (0,1);
\coordinate (1a) at (2,-2);
\coordinate (2a) at (-2,-2);
\coordinate (3a) at (-2,4);
\coordinate (4a) at (0,4);
\coordinate (12) at (-0.5,-1);
\coordinate (12a) at (-1.5,-2);
\coordinate (12b) at (-0.5,-2);
\coordinate (12') at (0.5,-1);
\coordinate (12'a) at (0.5,-2);
\coordinate (12'b) at (1.5,-2);
\coordinate (23) at (-1,1.5);
\coordinate (23a) at (-2,3.5);
\coordinate (23b) at (-2,2.5);
\coordinate (23') at (-1,0.5);
\coordinate (23'a) at (-2,1.5);
\coordinate (23'b) at (-2,0.5);
\coordinate (23'') at (-1,-0.5);
\coordinate (23''a) at (-2,-0.5);
\coordinate (23''b) at (-2,-1.5);
\coordinate (34) at (-0.5,1.5);
\coordinate (34a) at (-0.5,4);
\coordinate (34b) at (-1.75,4);
\coordinate (41) at (0.5,0);
\coordinate (41a) at (0.5,4);
\coordinate (41b) at (2,-1.5);
\draw (1) -- (2) -- (3) -- (4) -- (1);
\draw (0,0) -- (1a);
\draw (0,0) -- (2a);
\draw (0,0) -- (3a);
\draw (0,0) -- (4a);
\end{tikzpicture}
\end{minipage}
\begin{minipage}[c]{0.05\textwidth}
$(5b)$
\end{minipage}
\begin{minipage}[c]{0.18\textwidth}
\begin{tikzpicture}[scale=0.6]
\coordinate[fill,circle,inner sep=0.8pt] (0) at (0,0);
\coordinate[fill,circle,inner sep=0.8pt] (1) at (-1,-1);
\coordinate[fill,circle,inner sep=0.8pt] (2) at (1,-1);
\coordinate[fill,circle,inner sep=0.8pt] (3) at (1,0);
\coordinate[fill,circle,inner sep=0.8pt] (4) at (0,1);
\coordinate[fill,circle,inner sep=0.8pt] (5) at (-1,1);
\coordinate (1a) at (-2,-2);
\coordinate (2a) at (2,-2);
\coordinate (3a) at (2,0);
\coordinate (4a) at (0,2);
\coordinate (5a) at (-2,2);
\coordinate (12) at (-0.5,-1);
\coordinate (12a) at (-0.5,-2);
\coordinate (12b) at (-1.5,-2);
\coordinate (12') at (0.5,-1);
\coordinate (12'a) at (0.5,-2);
\coordinate (12'b) at (1.5,-2);
\coordinate (23) at (1,-0.5);
\coordinate (23a) at (2,-1.5);
\coordinate (23b) at (2,-0.5);
\coordinate (34) at (0.5,0.5);
\coordinate (34a) at (2,0.5);
\coordinate (34b) at (0.5,2);
\coordinate (45) at (-0.5,1);
\coordinate (45a) at (-0.5,2);
\coordinate (45b) at (-1.5,2);
\coordinate (51) at (-1,0.5);
\coordinate (51a) at (-2,1.5);
\coordinate (51b) at (-2,0.5);
\coordinate (51') at (-1,-0.5);
\coordinate (51'a) at (-2,-1.5);
\coordinate (51'b) at (-2,-0.5);
\draw (1) -- (2) -- (3) -- (4) -- (5) -- (1);
\draw (0,0) -- (1a);
\draw (0,0) -- (2a);
\draw (0,0) -- (3a);
\draw (0,0) -- (4a);
\draw (0,0) -- (5a);
\end{tikzpicture}
\end{minipage}\\[6mm]
\begin{minipage}[c]{0.05\textwidth}
$(4a)$
\end{minipage}
\begin{minipage}[c]{0.18\textwidth}
\begin{tikzpicture}[scale=0.6]
\coordinate[fill,circle,inner sep=0.8pt] (0) at (0,0);
\coordinate[fill,circle,inner sep=0.8pt] (1) at (1,-1);
\coordinate[fill,circle,inner sep=0.8pt] (2) at (-1,-1);
\coordinate[fill,circle,inner sep=0.8pt] (3) at (-1,1);
\coordinate[fill,circle,inner sep=0.8pt] (4) at (1,1);
\coordinate (1a) at (2,-2);
\coordinate (2a) at (-2,-2);
\coordinate (3a) at (-2,2);
\coordinate (4a) at (2,2);
\coordinate (12) at (0.5,-1);
\coordinate (12a) at (1.5,-2);
\coordinate (12b) at (0.5,-2);
\coordinate (12') at (-0.5,-1);
\coordinate (12'a) at (-0.5,-2);
\coordinate (12'b) at (-1.5,-2);
\coordinate (23) at (-1,0.5);
\coordinate (23a) at (-2,1.5);
\coordinate (23b) at (-2,0.5);
\coordinate (23') at (-1,-0.5);
\coordinate (23'a) at (-2,-0.5);
\coordinate (23'b) at (-2,-1.5);
\coordinate (34) at (-0.5,1);
\coordinate (34a) at (-0.5,2);
\coordinate (34b) at (-1.5,2);
\coordinate (34') at (0.5,1);
\coordinate (34'a) at (0.5,2);
\coordinate (34'b) at (1.5,2);
\coordinate (41) at (1,0.5);
\coordinate (41a) at (2,0.5);
\coordinate (41b) at (2,1.5);
\coordinate (41') at (1,-0.5);
\coordinate (41'a) at (2,-0.5);
\coordinate (41'b) at (2,-1.5);
\draw (1) -- (2) -- (3) -- (4) -- (1);
\draw (0,0) -- (1a);
\draw (0,0) -- (2a);
\draw (0,0) -- (3a);
\draw (0,0) -- (4a);
\end{tikzpicture}
\end{minipage}
\begin{minipage}[c]{0.07\textwidth}
$(4b)$
\end{minipage}
\begin{minipage}[c]{0.16\textwidth}
\begin{tikzpicture}[scale=0.5]
\coordinate[fill,circle,inner sep=0.8pt] (0) at (0,0);
\coordinate[fill,circle,inner sep=0.8pt] (1) at (1,-1);
\coordinate[fill,circle,inner sep=0.8pt] (2) at (-1,-1);
\coordinate[fill,circle,inner sep=0.8pt] (3) at (-1,2);
\coordinate[fill,circle,inner sep=0.8pt] (4) at (1,0);
\coordinate (1a) at (2,-2);
\coordinate (2a) at (-2,-2);
\coordinate (3a) at (-2,4);
\coordinate (4a) at (2,0);
\coordinate (12) at (-0.5,-1);
\coordinate (12a) at (-1.5,-2);
\coordinate (12b) at (-0.5,-2);
\coordinate (12') at (0.5,-1);
\coordinate (12'a) at (0.5,-2);
\coordinate (12'b) at (1.5,-2);
\coordinate (23) at (-1,1.5);
\coordinate (23a) at (-2,3.5);
\coordinate (23b) at (-2,2.5);
\coordinate (23') at (-1,0.5);
\coordinate (23'a) at (-2,1.5);
\coordinate (23'b) at (-2,0.5);
\coordinate (23'') at (-1,-0.5);
\coordinate (23''a) at (-2,-0.5);
\coordinate (23''b) at (-2,-1.5);
\coordinate (34) at (-0.5,1.5);
\coordinate (34a) at (-0.5,4);
\coordinate (34b) at (-1.75,4);
\coordinate (34') at (0.5,0.5);
\coordinate (34'a) at (0.5,4);
\coordinate (34'b) at (2,0.5);
\coordinate (41) at (1,-0.5);
\coordinate (41a) at (2,-0.5);
\coordinate (41b) at (2,-1.5);
\draw (1) -- (2) -- (3) -- (4) -- (1);
\draw (0,0) -- (1a);
\draw (0,0) -- (2a);
\draw (0,0) -- (3a);
\draw (0,0) -- (4a);
\end{tikzpicture}
\end{minipage}
\begin{minipage}[c]{0.06\textwidth}
$(4c)$
\end{minipage}
\begin{minipage}[c]{0.12\textwidth}
\begin{tikzpicture}[scale=0.375]
\coordinate[fill,circle,inner sep=0.8pt] (0) at (0,0);
\coordinate[fill,circle,inner sep=0.8pt] (1) at (-1,-1);
\coordinate[fill,circle,inner sep=0.8pt] (2) at (1,-1);
\coordinate[fill,circle,inner sep=0.8pt] (3) at (-1,3);
\coordinate (1a) at (-2,-2);
\coordinate (2a) at (2,-2);
\coordinate (3a) at (-2,6);
\coordinate (12) at (-0.5,-1);
\coordinate (12a) at (-1.5,-2);
\coordinate (12b) at (-0.5,-2);
\coordinate (12') at (0.5,-1);
\coordinate (12'a) at (0.5,-2);
\coordinate (12'b) at (1.5,-2);
\coordinate (23) at (0.5,0);
\coordinate (23a) at (2,-1.5);
\coordinate (23b) at (0.5,6);
\coordinate (23') at (-0.5,2);
\coordinate (23'a) at (-1.833,6);
\coordinate (23'b) at (-0.5,6);
\coordinate (31) at (-1,-0.5);
\coordinate (31a) at (-2,-1.5);
\coordinate (31b) at (-2,-0.5);
\coordinate (31') at (-1,0.5);
\coordinate (31'a) at (-2,0.5);
\coordinate (31'b) at (-2,1.5);
\coordinate (31'') at (-1,1.5);
\coordinate (31''a) at (-2,2.5);
\coordinate (31''b) at (-2,3.5);
\coordinate (31''') at (-1,2.5);
\coordinate (31'''a) at (-2,4.5);
\coordinate (31'''b) at (-2,5.5);
\draw (1) -- (2) -- (3) -- (1);
\draw (0,0) -- (1a);
\draw (0,0) -- (2a);
\draw (0,0) -- (3a);
\end{tikzpicture}
\end{minipage}
\begin{minipage}[c]{0.05\textwidth}
$(3a)$
\end{minipage}
\begin{minipage}[c]{0.23\textwidth}
\begin{tikzpicture}[scale=0.5]
\coordinate[fill,circle,inner sep=0.8pt] (0) at (0,0);
\coordinate[fill,circle,inner sep=0.8pt] (1) at (-1,-1);
\coordinate[fill,circle,inner sep=0.8pt] (2) at (2,-1);
\coordinate[fill,circle,inner sep=0.8pt] (3) at (-1,2);
\coordinate (1a) at (-2,-2);
\coordinate (2a) at (4,-2);
\coordinate (3a) at (-2,4);
\coordinate (12) at (-0.5,-1);
\coordinate (12a) at (-1.5,-2);
\coordinate (12b) at (-0.5,-2);
\coordinate (12') at (0.5,-1);
\coordinate (12'a) at (0.5,-2);
\coordinate (12'b) at (1.5,-2);
\coordinate (12'') at (1.5,-1);
\coordinate (12''a) at (2.5,-2);
\coordinate (12''b) at (3.5,-2);
\coordinate (23) at (-0.5,1.5);
\coordinate (23a) at (-0.5,4);
\coordinate (23b) at (-1.75,4);
\coordinate (23') at (0.5,0.5);
\coordinate (23'a) at (4,0.5);
\coordinate (23'b) at (0.5,4);
\coordinate (23'') at (1.5,-0.5);
\coordinate (23''a) at (4,-0.5);
\coordinate (23''b) at (4,-1.75);
\coordinate (31) at (-1,-0.5);
\coordinate (31a) at (-2,-1.5);
\coordinate (31b) at (-2,-0.5);
\coordinate (31') at (-1,0.5);
\coordinate (31'a) at (-2,0.5);
\coordinate (31'b) at (-2,1.5);
\coordinate (31'') at (-1,1.5);
\coordinate (31''a) at (-2,2.5);
\coordinate (31''b) at (-2,3.5);
\draw (1) -- (2) -- (3) -- (1);
\draw (0,0) -- (1a);
\draw (0,0) -- (2a);
\draw (0,0) -- (3a);
\end{tikzpicture}
\end{minipage}
\caption{Fans of Gorenstein toric Fano surfaces and the (reflexive) polytopes $\Delta(X^{tor})$ spanned by their rays.}
\label{fig:16}
\end{figure}

\subsection{Unfolding of the spanning polytope}
\label{S:unfold}

For a vertex $v$ of $\Delta$, the \textit{kink} of $\Delta$ at $v$ is defined as $k_v=|\braket{m_1^\bot,m_2}|=|\textup{det}(m_1|m_2)|$, where $m_1,m_2\in\mathbb{Z}^2$ are the primitive integral tangent vectors pointing from $v$ to its adjacent edges. Choose a vertex of $\Delta$ and fix an ordering (counterclockwise). This gives a numbering of the vertices $v=v_1,v_2,\ldots,v_r$. Let $k_i$ be the kink of $v_i$ and let $l_i$ be the affine length of the edge connecting $v_i$ with $v_{i+1\textup{ mod }r}$. Then $\vec{k}=(k_1,\ldots,k_r)$ and $\vec{l}=(l_1,\ldots,l_r)$ determine $\Delta$ up to $SL_2(\mathbb{Z})$-transformations, hence are uniquely determined by $X^{tor}$. 

\begin{defi}
Let $\Delta_\infty(X^{tor})\subset\mathbb{R}^2$ be the unique unbounded convex polytope with
\begin{compactenum}[(1)]
\item $\Delta_\infty(X^{tor})$ is contained in the lower half plane;
\item $\Delta_\infty(X^{tor})$ has the line segment $\textup{Conv}\{(0,0),(l_1,0)\}$ as an edge;
\item $\Delta_\infty(X^{tor})$ has vertices whose kinks are an infinite repetition of $k_1,\ldots,k_r$;
\item The kink at $(0,0)$ is $k_1$ and the kink at $(l_1,0)$ is $k_2$;
\item The edge connecting a vertex with kink $k_i$ and a vertex with kink $k_{i+1 \textup{ mod } r}$ has affine length $l_i$.
\end{compactenum}
We call $\Delta_\infty(X^{tor})$ the \textit{unfolding of the spanning polytope} of $X^{tor}$. A \textit{fundamental domain} for $\Delta_\infty(X^{tor})$ is a set $I \times \mathbb{R} \subset\mathbb{R}^2$ for $I$ a half-open interval of length $\sum_{i=1}^r l_i$.
\end{defi}

\begin{expl}
\label{expl:unfold}
See Figure \ref{fig:unfold} for part of $\Delta_\infty(X^{tor})$ of the following cases. The part of $\Delta_\infty(X^{tor})$ outside a chosen fundamental domain $I\times\mathbb{R}$ is shown dashed.
\begin{itemize}
\item[(9)] For $X=X^{tor}=\mathbb{P}^2$ we have $\vec{k}=(3,3,3)$, $\vec{l}=(1,1,1)$ and $I=[-1,2)$.
\item[(8)] For $X=X^{tor}=\mathbb{F}_1$ we have $\vec{k}=(1,2,3,2)$, $\vec{l}=(1,1,1,1)$ and $I=[-2,2)$.
\item[(3a)] For $X$ a smooth cubic surface and $X^{tor}=\mathbb{P}^2/\mathbb{Z}_3$ we have $\vec{k}=(1,1,1)$, $\vec{l}=(3,3,3)$ and $I=[-3,6)$.
\end{itemize}
\end{expl}

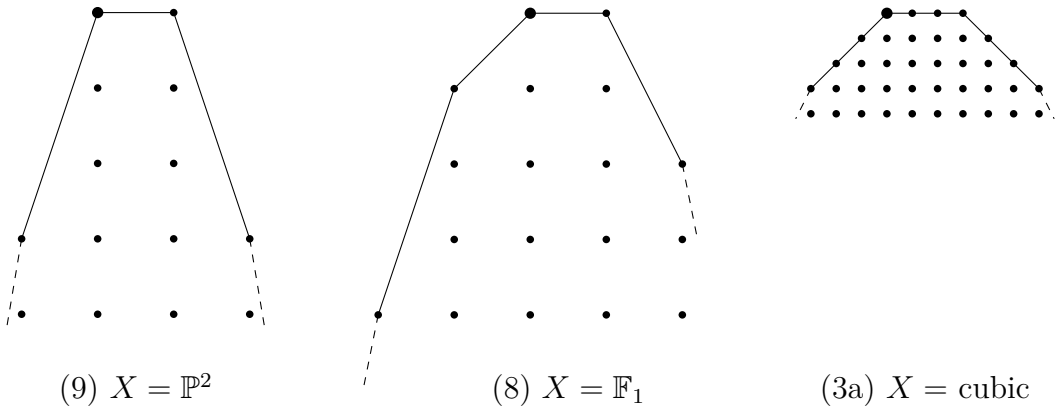
\begin{figure}[h!]
\centering
\begin{tikzpicture}
\draw (-1,-3) -- (0,0) -- (1,0) -- (2,-3);
\draw[dashed] (2,-3) -- (2.2,-4.2);
\draw[dashed] (-1,-3) -- (-1.2,-4.2);
\draw (0.5,-5) node{(9) $X=\mathbb{P}^2$};
\draw (0,0) node[fill,circle,inner sep=1.5pt]{};
\draw (0,-1) node[fill,circle,inner sep=1pt]{};
\draw (0,-2) node[fill,circle,inner sep=1pt]{};
\draw (0,-3) node[fill,circle,inner sep=1pt]{};
\draw (0,-4) node[fill,circle,inner sep=1pt]{};
\draw (1,0) node[fill,circle,inner sep=1pt]{};
\draw (1,-1) node[fill,circle,inner sep=1pt]{};
\draw (1,-2) node[fill,circle,inner sep=1pt]{};
\draw (1,-3) node[fill,circle,inner sep=1pt]{};
\draw (1,-4) node[fill,circle,inner sep=1pt]{};
\draw (-1,-3) node[fill,circle,inner sep=1pt]{};
\draw (-1,-4) node[fill,circle,inner sep=1pt]{};
\draw (2,-3) node[fill,circle,inner sep=1pt]{};
\draw (2,-4) node[fill,circle,inner sep=1pt]{};
\end{tikzpicture}
\hspace{1cm}
\begin{tikzpicture}
\draw (-2,-4) -- (-1,-1) -- (0,0) -- (1,0) -- (2,-2);
\draw[dashed] (2,-2) -- (2.2,-3);
\draw[dashed] (-2,-4) -- (-2.2,-5);
\draw (0.5,-5) node{(8) $X=\mathbb{F}_1$};
\draw (0,0) node[fill,circle,inner sep=1.5pt]{};
\draw (0,-1) node[fill,circle,inner sep=1pt]{};
\draw (0,-2) node[fill,circle,inner sep=1pt]{};
\draw (0,-3) node[fill,circle,inner sep=1pt]{};
\draw (0,-4) node[fill,circle,inner sep=1pt]{};
\draw (1,0) node[fill,circle,inner sep=1pt]{};
\draw (1,-1) node[fill,circle,inner sep=1pt]{};
\draw (1,-2) node[fill,circle,inner sep=1pt]{};
\draw (1,-3) node[fill,circle,inner sep=1pt]{};
\draw (1,-4) node[fill,circle,inner sep=1pt]{};
\draw (-1,-1) node[fill,circle,inner sep=1pt]{};
\draw (-1,-2) node[fill,circle,inner sep=1pt]{};
\draw (-1,-3) node[fill,circle,inner sep=1pt]{};
\draw (-1,-4) node[fill,circle,inner sep=1pt]{};
\draw (-2,-4) node[fill,circle,inner sep=1pt]{};
\draw (2,-2) node[fill,circle,inner sep=1pt]{};
\draw (2,-3) node[fill,circle,inner sep=1pt]{};
\draw (2,-4) node[fill,circle,inner sep=1pt]{};
\end{tikzpicture}
\hspace{1cm}
\begin{tikzpicture}
\draw (-1,-1) -- (0,0) -- (1,0) -- (2,-1);
\draw[dashed] (2,-1) -- (2.2,-1.4);
\draw[dashed] (-1,-1) -- (-1.2,-1.4);
\draw (0.5,-5) node{(3a) $X=$ cubic};
\draw (0,0) node[fill,circle,inner sep=1.5pt]{};
\draw (0,-1/3) node[fill,circle,inner sep=1pt]{};
\draw (0,-2/3) node[fill,circle,inner sep=1pt]{};
\draw (0,-3/3) node[fill,circle,inner sep=1pt]{};
\draw (0,-4/3) node[fill,circle,inner sep=1pt]{};
\draw (1/3,0) node[fill,circle,inner sep=1pt]{};
\draw (1/3,-1/3) node[fill,circle,inner sep=1pt]{};
\draw (1/3,-2/3) node[fill,circle,inner sep=1pt]{};
\draw (1/3,-3/3) node[fill,circle,inner sep=1pt]{};
\draw (1/3,-4/3) node[fill,circle,inner sep=1pt]{};
\draw (2/3,0) node[fill,circle,inner sep=1pt]{};
\draw (2/3,-1/3) node[fill,circle,inner sep=1pt]{};
\draw (2/3,-2/3) node[fill,circle,inner sep=1pt]{};
\draw (2/3,-3/3) node[fill,circle,inner sep=1pt]{};
\draw (2/3,-4/3) node[fill,circle,inner sep=1pt]{};
\draw (1,0) node[fill,circle,inner sep=1pt]{};
\draw (1,-1/3) node[fill,circle,inner sep=1pt]{};
\draw (1,-2/3) node[fill,circle,inner sep=1pt]{};
\draw (1,-3/3) node[fill,circle,inner sep=1pt]{};
\draw (1,-4/3) node[fill,circle,inner sep=1pt]{};
\draw (-1/3,-1/3) node[fill,circle,inner sep=1pt]{};
\draw (-1/3,-2/3) node[fill,circle,inner sep=1pt]{};
\draw (-1/3,-3/3) node[fill,circle,inner sep=1pt]{};
\draw (-1/3,-4/3) node[fill,circle,inner sep=1pt]{};
\draw (-2/3,-2/3) node[fill,circle,inner sep=1pt]{};
\draw (-2/3,-3/3) node[fill,circle,inner sep=1pt]{};
\draw (-2/3,-4/3) node[fill,circle,inner sep=1pt]{};
\draw (-1,-1) node[fill,circle,inner sep=1pt]{};
\draw (-1,-4/3) node[fill,circle,inner sep=1pt]{};
\draw (1+1/3,-1/3) node[fill,circle,inner sep=1pt]{};
\draw (1+1/3,-2/3) node[fill,circle,inner sep=1pt]{};
\draw (1+1/3,-3/3) node[fill,circle,inner sep=1pt]{};
\draw (1+1/3,-4/3) node[fill,circle,inner sep=1pt]{};
\draw (1+2/3,-2/3) node[fill,circle,inner sep=1pt]{};
\draw (1+2/3,-3/3) node[fill,circle,inner sep=1pt]{};
\draw (1+2/3,-4/3) node[fill,circle,inner sep=1pt]{};
\draw (1+1,-1) node[fill,circle,inner sep=1pt]{};
\draw (1+1,-4/3) node[fill,circle,inner sep=1pt]{};
\end{tikzpicture}
\caption{Unfolding of the spanning polytope $\Delta_\infty(X^{tor})$ for some $X$.}
\label{fig:unfold}
\end{figure}

\subsection{Initial scattering diagram}
\label{S:initial}

\begin{defi}
A \textit{scattering diagram} (or \textit{wall structure}) $\mathscr{S}$ is a collection of rays $\mathfrak{d}=b_{\mathfrak{d}}+\mathbb{R}_{\geq 0}m_{\mathfrak{d}}$, for $b_{\mathfrak{d}}\in\mathbb{R}^2$ and $m_{\mathfrak{d}}\in\mathbb{Z}^2$ primitive (i.e., $\textup{gcd}(\{m_i\})=1$), with attached functions $f_\mathfrak{d}\in\mathbb{C}\llbracket t\rrbracket [z^{\pm m_{\mathfrak{d}}}]$ such that $f_\mathfrak{d}$ has $t$-constant term $1$ and $f_\mathfrak{d}\neq 1$.
\end{defi}

Let $(X,D)$ be a very ample log Calabi-Yau pair and let $X^{tor}$ be a toric model of $X$. For an edge $e$ of $\Delta_\infty(X^{tor})$ let $l_e$ be its affine length and let $p_e$ be some interior point (e.g. the middle point). Let $v_e^+$ and $v_e^-$ be its adjacent vertices, and let $m_e^\pm\in\mathbb{Z}^2$ be the primitive integral vector pointing from $p_e$ in the direction of $v_e^\pm$. Let $\mathscr{S}_0(X^{tor})$ be the scattering diagram that has for each edge $e$ of $\Delta_\infty(X^{tor})$ two rays
\[ \mathfrak{d}_e^\pm = p_e + \mathbb{R}_{\geq 0}m_e^\pm, \quad f_{\mathfrak{d}_e^\pm}=(1+tz^{m_e^\pm})^{l_e}. \]
Note that $\mathscr{S}_0(X^{tor})$ contains finitely many rays with base in a given fundamental domain.

\begin{expl}
Figure \ref{fig:scat} shows $\mathscr{S}_0(X^{tor})$ for the cases from Example \ref{expl:unfold}. The points $p_e$ are indicated by crosses. Later they will correspond to affine singularities (see \S\ref{S:XD}). The parts of rays outside the fundamental domain chosen in Example \ref{expl:unfold} are shown dashed. The function attached to a ray $\mathfrak{d}$ with primitive direction $m\in\mathbb{Z}^2$ is $1+tz^m$ for $\mathbb{P}^2$ and $\mathbb{F}_1$ and $(1+tz^m)^3$ for the cubic surface.
\end{expl}

\begin{figure}[h!]
\centering
\begin{tikzpicture}
\draw (-1,-3) -- (0,0) -- (1,0) -- (2,-3);
\draw[dashed] (2,-3) -- (2.2,-4.2);
\draw[dashed] (-1,-3) -- (-1.2,-4.2);
\draw (0.5,-5) node{(9) $=\mathbb{P}^2$};
\draw (0.5,0) node[fill,cross,inner sep=2pt]{} -- (2,0);
\draw[->,dashed] (2,0) -- (2.2,0);
\draw (0.5,0) node[fill,cross,inner sep=2pt]{} -- (-1,0);
\draw[->,dashed] (-1,0) -- (-1.2,0);
\draw[->] (-0.5,-1.5) node[fill,cross,inner sep=2pt,rotate=71.565]{} -- (1.2,3.6);
\draw (-0.5,-1.5) node[fill,cross,inner sep=2pt,rotate=71.57]{} -- (-1,-3);
\draw[->,dashed] (-1,-3) -- (-1.2,-3.6);
\draw[->] (1.5,-1.5) node[fill,cross,inner sep=2pt,rotate=-71.565]{} -- (-0.2,3.6);
\draw (1.5,-1.5) node[fill,cross,inner sep=2pt,rotate=-71.57]{} -- (2,-3);
\draw[->,dashed] (2,-3) -- (2.2,-3.6);
\draw[->] (-0.5,-1.5) node[fill,cross,inner sep=2pt,rotate=71.565]{} -- (1.2,3.6);
\draw[->] (-1,-3) -- (0.1,3.6);
\draw[->] (2,-3) -- (0.9,3.6);
\draw[dashed] (-1.2,-1.8) -- (-1,0);
\draw[->] (-1,0) -- (-0.6,3.6);
\draw[dashed] (2.2,-1.8) -- (2,0);
\draw[->] (2,0) -- (1.6,3.6);
\end{tikzpicture}
\hspace{1cm} 
\begin{tikzpicture}
\draw (-2,-4) -- (-1,-1) -- (0,0) -- (1,0) -- (2,-2);
\draw[dashed] (2,-2) -- (2.2,-3);
\draw[dashed] (-2,-4) -- (-2.2,-5);
\draw (0.5,-5) node{(8) $=\mathbb{F}_1$};
\draw (0.5,0) node[fill,cross,inner sep=2pt]{} -- (2,0);
\draw[->,dashed] (2,0) -- (2.2,0);
\draw (0.5,0) node[fill,cross,inner sep=2pt]{} -- (-2,0);
\draw[->,dashed] (-2,0) -- (-2.2,0);
\draw (-0.5,-0.5) node[fill,cross,inner sep=2pt,rotate=45]{} -- (2,2);
\draw[->,dashed] (2,2) -- (2.2,2.2);
\draw (-0.5,-0.5) node[fill,cross,inner sep=2pt,rotate=45]{} -- (-2,-2);
\draw[->,dashed] (-2,-2) -- (-2.2,-2.2);
\draw[->] (-1.5,-2.5) node[fill,cross,inner sep=2pt,rotate=71.565]{} -- (8/15,3.6);
\draw (-1.5,-2.5) node[fill,cross,inner sep=2pt,rotate=71.565]{} -- (-2,-4);
\draw[->,dashed] (-2,-4) -- (-2.2,-4.6);
\draw[->] (1.5,-1) node[fill,cross,inner sep=2pt,rotate=-63.435]{} -- (-0.8,3.6);
\draw (1.5,-1) node[fill,cross,inner sep=2pt,rotate=-63.435]{} -- (2,-2);
\draw[->,dashed] (2,-2) -- (2.2,-2.4);
\draw[->] (-2,-4) -- (-11/15,3.6);
\draw[->] (2,-2) -- (0.88,3.6);
\draw[dashed] (-2.2,-3.6) -- (-2,-2);
\draw[->] (-2,-2) -- (-1.3,3.6);
\draw[dashed] (2.2,-2.2) -- (2,-1);
\draw[->] (2,-1) -- (37/30,3.6);
\draw[dashed] (-2.2,-1.8) -- (-2,0);
\draw[->] (-2,0) -- (-1.6,3.6);
\draw[dashed] (2.2,1.4) -- (2,3);
\draw[->] (2,3) -- (1.925,3.6);
\end{tikzpicture}
\hspace{1cm} 
\begin{tikzpicture}
\draw (-1,-1) -- (0,0) -- (1,0) -- (2,-1);
\draw[dashed] (2,-1) -- (2.2,-1.4);
\draw[dashed] (-1,-1) -- (-1.2,-1.4);
\draw (0.5,-5) node{(3a) $=$ cubic};
\draw (0.5,0) node[fill,cross,inner sep=2pt]{} -- (2,0);
\draw[->,dashed] (2,0) -- (2.2,0);
\draw (0.5,0) node[fill,cross,inner sep=2pt]{} -- (-1,0);
\draw[->,dashed] (-1,0) -- (-1.2,0);
\draw (-0.5,-0.5) node[fill,cross,inner sep=2pt,rotate=45]{} -- (2,2);
\draw[->,dashed] (2,2) -- (2.2,2.2);
\draw (-0.5,-0.5) node[fill,cross,inner sep=2pt,rotate=45]{} -- (-1,-1);
\draw[->,dashed] (-1,-1) -- (-1.2,-1.2);
\draw (1.5,-0.5) node[fill,cross,inner sep=2pt,rotate=45]{} -- (-1,2);
\draw[->,dashed] (-1,2) -- (-1.2,2.2);
\draw (1.5,-0.5) node[fill,cross,inner sep=2pt,rotate=45]{} -- (2,-1);
\draw[->,dashed] (2,-1) -- (2.2,-1.2);
\draw[->] (-1,-1) -- (1.3,3.6);
\draw[->] (2,-1) -- (-0.3,3.6);
\draw[dashed] (-1.2,-0.6) -- (-1,0);
\draw[->] (-1,0) -- (0.2,3.6);
\draw[dashed] (2.2,-0.6) -- (2,0);
\draw[->] (2,0) -- (0.8,3.6);
\draw[dashed] (-1.2,1.2) -- (-1,2);
\draw[->] (-1,2) -- (-0.6,3.6);
\draw[dashed] (2.2,1.2) -- (2,2);
\draw[->] (2,2) -- (1.6,3.6);
\end{tikzpicture}
\caption{The initial scattering diagram $\mathscr{S}_0(X^{tor})$.}
\label{fig:scat}
\end{figure}
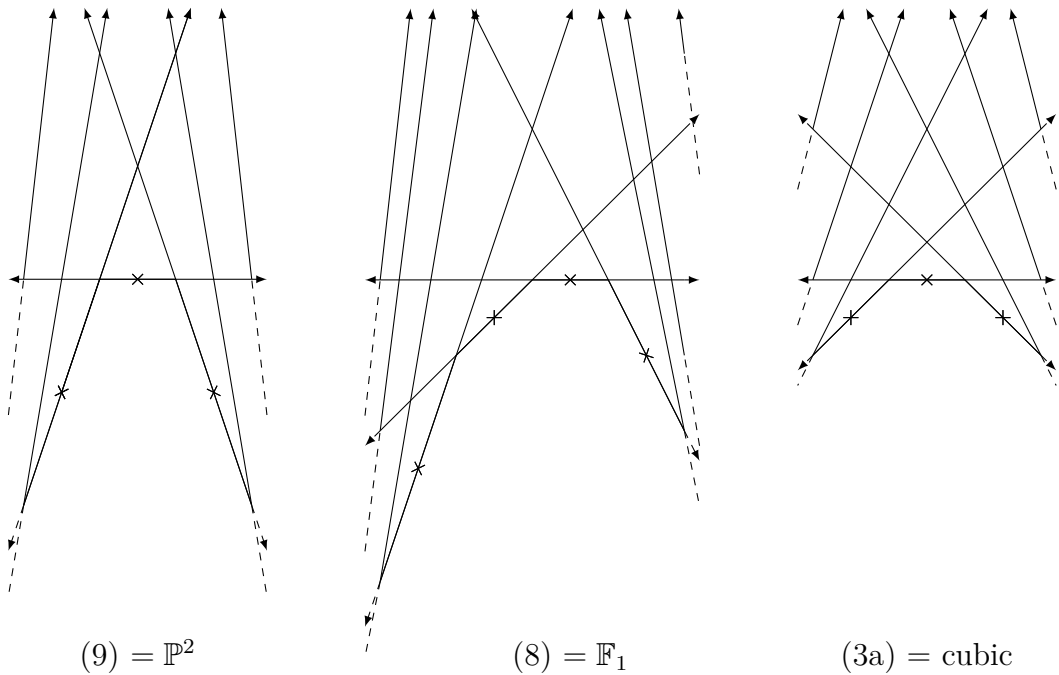

\subsection{Scattering}
\label{S:scat}

There is an algorithm called \textit{scattering} which produces from a scattering diagram $\mathscr{S}_0$ a bigger scattering diagram $\mathscr{S}_\infty$ containing $\mathscr{S}_0$. It was introduced by Kontsevich-Soibelman \cite{KS} and used by Gross-Siebert \cite{GSreconstruction} to construct degenerations of Calabi-Yau manifolds. The mirror to a log Calabi-Yau pair $(X,D)$ is a Landau-Ginzburg model $W : \check{X} \rightarrow \mathbb{C}$. The superpotential $W$ was constructed in \cite{CPS} using objects on scattering diagrams called \textit{broken lines}. This gives a precise description of the $B$-model for Calabi-Yau manifolds or Landau-Ginzburg models. On the $A$-model side scattering diagrams can be interpreted in terms of curve counts (log Gromov-Witten invariants) of Calabi-Yau manifolds or log Calabi-Yau pairs. This was established for toric boundary divisors in \cite{GPS}, for Looijenga pairs in \cite{BBG} and for smooth divisors in \cite{Gra}. This note deals with the last case.

Let $\mathscr{S}_0$ be a scattering diagram such that all rays have the same base $p\in\mathbb{R}^2$. Let $\gamma$ be a simple loop around $p$. This gives an ordering of the rays $\mathfrak{d}_1,\ldots,\mathfrak{d}_r$. For a ray $\mathfrak{d}$ with function $f_{\mathfrak{d}}$ define a $\mathbb{C}\llbracket t\rrbracket$-automorphism
\[ \theta_\mathfrak{d} : \mathbb{C}\llbracket t\rrbracket [x^{\pm 1},y^{\pm 1}], z^m \mapsto f_\mathfrak{d}^{\braket{n_{\mathfrak{d}},m}}z^m, \]
where $n_{\mathfrak{d}}$ is the unique primitive normal vector to $\mathfrak{d}$ that evaluates positively on $\gamma$. Write $\theta_{\mathscr{S},\gamma} = \theta_{\mathfrak{d}_r} \circ \ldots \circ \theta_{\mathfrak{d}_1}$. We call $\mathscr{S}$ \textit{consistent} to order $k$ if $\theta_{\mathscr{S},\gamma} \equiv 1 \textup{ mod } t^{k+1}$.

Note that $\mathscr{S}_0$ is trivially consistent to order $0$. From a scattering diagram $\mathscr{S}_k$ consistent to order $k$ we construct a scattering diagram $\mathscr{S}_{k+1}$ as follows. 

By \cite{KS} or \cite{GPS}, Theorem 1.4, we can write,
\[ \theta_{\mathscr{S}_k,\gamma}(z^m) = \left(1+\sum_{i=1}^s a_it^{k+1}z^{m_i}\right)^{-\braket{n_i,m}}z^m \textup{ mod } t^{k+2}. \]
Here $n_i$ is the unique primitive normal vector to $m_i$ that evaluates positively on $\gamma$. For each term $a_it^{k+1}z^{m_i}$ we add a ray $\mathfrak{d} = p+\mathbb{R}_{\geq 0}m_i$ with function $f_{\mathfrak{d}} = 1+a_it^{k+1}z^{m_i}$ to $\mathscr{S}_k$ to obtain a new scattering diagram $\mathscr{S}_{k+1}$. By construction $\mathscr{S}_{k+1}$ is consistent to order $k+1$. Let $\mathscr{S}_\infty$ be the limit $k\rightarrow\infty$. This may contain infinitely many rays.

Now let $\mathscr{S}_0$ be any scattering diagram. If two or more rays intersect in a point $p$, then $\mathscr{S}_0$ can be localized at $p$. This means we consider only rays through $p$ and split a ray into two rays (with the same function) if $p$ is not its base. In this way we get a diagram $\mathscr{S}_p$ all whose rays have base $p$. We perform the scattering procedure above locally at each intersection of two or more rays to obtain $\mathscr{S}_k$ and $\mathscr{S}_\infty$.

\subsection{The main theorem}
\label{S:scat}

Let $(X,D)$ be a very ample log Calabi-Yau pair and let $X^{tor}$ be a toric model of $X$. Let $\mathscr{S}_0(X^{tor})$ be the initial scattering diagram constructed in \S\ref{S:initial} and let $\mathscr{S}_\infty(X^{tor})$ be the consistent diagram obtained from the scattering procedure from \S\ref{S:scat}. Write
\[ f_{\textup{out}}=\prod_{\mathfrak{d} : m_\mathfrak{d}=(0,1)} f_{\mathfrak{d}}|_{t=1}, \]
where the sum is over all rays $\mathfrak{d}$ of $\mathscr{S}_\infty(X^{tor})$ going upwards and we set $t=1$.

\begin{thm}[\cite{Gra}, Theorem 2]
\label{thm:main}
We have
\[ \textup{log }f_{\textup{out}} = \sum_\beta (D.\beta) R_\beta(X,D) y^{D.\beta}, \]
where the sum is over all effective curve classes of $X$ and $D.\beta$ is the intersection multiplicity of a curve of class $\beta$ with $D$.
\end{thm}

\begin{expl}
\label{expl:scat}
Figure \ref{fig:scatP2} shows the scattering diagram $\mathscr{S}_{(18)}(\mathbb{P}^2)$, rotated for reasons of space. The colors correspond to different $y$-orders of the rays. The sage code explained in \S\ref{S:sage} gives
\[ f_{\textup{out}} = (1+9y^3)^3 (1+72y^6)^3(1+36y^6)^3(1-78y^9)^3(1+81y^9)^3(1+243y^9)^6 + \mathcal{O}(y^{12}), \]
and thus
\[ \textup{log }f_{\textup{out}} = 3 \cdot 9 \cdot y^3 + 6\cdot \frac{135}{4} \cdot y^6 + 9 \cdot 244 \cdot y^9 + \mathcal{O}(y^{12}). \]
This gives the correct log Gromov-Witten invariants $R_d(\mathbb{P}^2,E)$ from Example \ref{expl:R}.
\end{expl}

\begin{figure}[h!]
\renewcommand\thefigure{1}
\centering
\begin{tikzpicture}[xscale=0.8,yscale=1,define rgb/.code={\definecolor{mycolor}{RGB}{#1}}, rgb color/.style={define rgb={#1},mycolor}]
\clip (-3,-1) rectangle (15,2);
\draw[->,rgb color={255,0,0}] (0.000,0.500) -- (0.000,6.50);
\draw[->,rgb color={255,0,0}] (0.000,0.500) -- (0.000,-5.50);
\draw[->,rgb color={255,0,0}] (-1.50,-0.500) -- (15.0,5.00);
\draw[->,rgb color={255,0,0}] (-1.50,1.50) -- (15.0,-4.00);
\draw[->,rgb color={255,0,0}] (-1.50,1.50) -- (-16.5,6.50);
\draw[->,rgb color={255,0,0}] (-1.50,-0.500) -- (-16.5,-5.50);
\draw[->,rgb color={255,0,0}] (-6.00,-1.50) -- (15.0,2.00);
\draw[->,rgb color={255,0,0}] (-6.00,2.50) -- (15.0,-1.00);
\draw[->,rgb color={255,0,0}] (-6.00,2.50) -- (-30.0,6.50);
\draw[->,rgb color={255,0,0}] (-6.00,-1.50) -- (-30.0,-5.50);
\draw[->,rgb color={255,0,0}] (-13.5,-2.50) -- (15.0,0.667);
\draw[->,rgb color={255,0,0}] (-13.5,3.50) -- (15.0,0.333);
\draw[->,rgb color={255,0,0}] (-13.5,3.50) -- (-40.5,6.50);
\draw[->,rgb color={255,0,0}] (-13.5,-2.50) -- (-40.5,-5.50);
\draw[->,rgb color={255,0,0}] (-24.0,-3.50) -- (15.0,-0.250);
\draw[->,rgb color={255,0,0}] (-24.0,4.50) -- (15.0,1.25);
\draw[->,rgb color={255,0,0}] (-24.0,4.50) -- (-48.0,6.50);
\draw[->,rgb color={255,0,0}] (-24.0,-3.50) -- (-48.0,-5.50);
\draw[->,rgb color={255,0,0}] (-37.5,-4.50) -- (15.0,-1.00);
\draw[->,rgb color={255,0,0}] (-37.5,5.50) -- (15.0,2.00);
\draw[->,rgb color={255,0,0}] (-37.5,5.50) -- (-52.5,6.50);
\draw[->,rgb color={255,0,0}] (-37.5,-4.50) -- (-52.5,-5.50);
\draw[->,rgb color={255,0,0}] (-54.0,-5.50) -- (15.0,-1.67);
\draw[->,rgb color={255,0,0}] (-54.0,6.50) -- (15.0,2.67);
\draw[->,red] (0.000,0.000) -- (15.0,0.000);
\draw[->,red] (-3.00,-1.00) -- (15.0,-1.00);
\draw[->,red] (0.000,1.00) -- (15.0,1.00);
\draw[->,red] (-3.00,2.00) -- (15.0,2.00);
\draw[->,red] (-45.0,-5.00) -- (15.0,-5.00);
\draw[->,red] (-30.0,-4.00) -- (15.0,-4.00);
\draw[->,red] (-18.0,-3.00) -- (15.0,-3.00);
\draw[->,red] (-9.00,-2.00) -- (15.0,-2.00);
\draw[->,red] (-9.00,3.00) -- (15.0,3.00);
\draw[->,red] (-18.0,4.00) -- (15.0,4.00);
\draw[->,red] (-30.0,5.00) -- (15.0,5.00);
\draw[->,red] (-45.0,6.00) -- (15.0,6.00);
\draw[->,black] (1.50,0.500) -- (15.0,0.500);
\draw[->,black] (-4.50,-1.50) -- (15.0,-1.50);
\draw[->,black] (0.000,-0.500) -- (15.0,-0.500);
\draw[->,black] (0.000,1.50) -- (15.0,1.50);
\draw[->,black] (-4.50,2.50) -- (15.0,2.50);
\draw[->,black] (-22.5,-3.50) -- (15.0,-3.50);
\draw[->,black] (-12.0,-2.50) -- (15.0,-2.50);
\draw[->,black] (-12.0,3.50) -- (15.0,3.50);
\draw[->,black] (-22.5,4.50) -- (15.0,4.50);
\draw[->,rgb color={255,169,0}] (0.000,0.000) -- (15.0,2.50);
\draw[->,rgb color={255,169,0}] (-18.0,-3.00) -- (15.0,-0.800);
\draw[->,rgb color={255,152,0}] (-9.00,-2.00) -- (15.0,0.000);
\draw[->,rgb color={255,132,0}] (-3.00,-1.00) -- (15.0,1.00);
\draw[->,rgb color={255,169,0}] (0.000,1.00) -- (15.0,6.00);
\draw[->,rgb color={255,132,0}] (-3.00,2.00) -- (-3.00,6.50);
\draw[->,rgb color={255,152,0}] (-9.00,3.00) -- (-19.5,6.50);
\draw[->,rgb color={255,169,0}] (-18.0,4.00) -- (-33.0,6.50);
\draw[->,rgb color={255,169,0}] (0.000,0.000) -- (15.0,-5.00);
\draw[->,rgb color={255,169,0}] (-18.0,-3.00) -- (-33.0,-5.50);
\draw[->,rgb color={255,152,0}] (-9.00,-2.00) -- (-19.5,-5.50);
\draw[->,rgb color={255,132,0}] (-3.00,-1.00) -- (-3.00,-5.50);
\draw[->,rgb color={255,169,0}] (0.000,1.00) -- (15.0,-1.50);
\draw[->,rgb color={255,132,0}] (-3.00,2.00) -- (15.0,0.000);
\draw[->,rgb color={255,152,0}] (-9.00,3.00) -- (15.0,1.00);
\draw[->,rgb color={255,169,0}] (-18.0,4.00) -- (15.0,1.80);
\draw[->,rgb color={255,152,0}] (3.00,0.000) -- (15.0,1.33);
\draw[->,rgb color={255,169,0}] (0.000,-1.00) -- (15.0,0.250);
\draw[->,rgb color={255,152,0}] (3.00,1.00) -- (15.0,3.00);
\draw[->,rgb color={255,169,0}] (0.000,2.00) -- (13.5,6.50);
\draw[->,rgb color={255,152,0}] (3.00,0.000) -- (15.0,-2.00);
\draw[->,rgb color={255,169,0}] (0.000,-1.00) -- (13.5,-5.50);
\draw[->,rgb color={255,152,0}] (3.00,1.00) -- (15.0,-0.333);
\draw[->,rgb color={255,169,0}] (0.000,2.00) -- (15.0,0.750);
\draw[->,rgb color={255,169,0}] (3.00,0.000) -- (15.0,-1.33);
\draw[->,rgb color={255,169,0}] (3.00,1.00) -- (15.0,0.000);
\draw[->,rgb color={255,152,0}] (1.50,0.500) -- (15.0,-1.00);
\draw[->,rgb color={255,169,0}] (0.000,-0.500) -- (15.0,-3.00);
\draw[->,rgb color={255,152,0}] (0.000,1.50) -- (15.0,0.250);
\draw[->,rgb color={255,152,0}] (1.50,0.500) -- (15.0,2.00);
\draw[->,rgb color={255,169,0}] (0.000,-0.500) -- (15.0,0.750);
\draw[->,rgb color={255,152,0}] (0.000,1.50) -- (15.0,4.00);
\draw[->,blue] (2.00,0.333) -- (15.0,0.333);
\draw[->,blue] (2.00,0.667) -- (15.0,0.667);
\draw[->,brown] (2.40,0.200) -- (15.0,0.200);
\draw[->,brown] (0.000,-0.800) -- (15.0,-0.800);
\draw[->,brown] (1.80,1.20) -- (15.0,1.20);
\draw[->,brown] (2.40,0.800) -- (15.0,0.800);
\draw[->,brown] (1.80,-0.200) -- (15.0,-0.200);
\draw[->,brown] (0.000,1.80) -- (15.0,1.80);
\draw[->,rgb color={255,169,0}] (3.00,0.000) -- (15.0,1.00);
\draw[->,rgb color={255,169,0}] (3.00,1.00) -- (15.0,2.33);
\draw[->,gray] (2.50,0.167) -- (15.0,0.167);
\draw[->,gray] (2.50,0.833) -- (15.0,0.833);
\draw[->,blue] (-13.0,-2.67) -- (15.0,-2.67);
\draw[->,blue] (-5.00,-1.67) -- (15.0,-1.67);
\draw[->,blue] (0.000,-0.667) -- (15.0,-0.667);
\draw[->,blue] (1.00,1.33) -- (15.0,1.33);
\draw[->,blue] (-3.00,2.33) -- (15.0,2.33);
\draw[->,blue] (-10.0,3.33) -- (15.0,3.33);
\draw[->,rgb color={255,169,0}] (2.00,0.333) -- (15.0,-0.750);
\draw[->,blue] (-10.0,-2.33) -- (15.0,-2.33);
\draw[->,blue] (-3.00,-1.33) -- (15.0,-1.33);
\draw[->,blue] (1.00,-0.333) -- (15.0,-0.333);
\draw[->,blue] (0.000,1.67) -- (15.0,1.67);
\draw[->,blue] (-5.00,2.67) -- (15.0,2.67);
\draw[->,blue] (-13.0,3.67) -- (15.0,3.67);
\draw[->,rgb color={255,169,0}] (2.00,0.667) -- (15.0,1.75);
\draw[->,rgb color={255,152,0}] (0.000,0.000) -- (15.0,3.33);
\draw[->,rgb color={255,169,0}] (-3.00,-1.00) -- (15.0,1.40);
\draw[->,rgb color={255,152,0}] (0.000,1.00) -- (8.25,6.50);
\draw[->,rgb color={255,169,0}] (-3.00,2.00) -- (-9.75,6.50);
\draw[->,rgb color={255,132,0}] (0.000,0.000) -- (15.0,1.67);
\draw[->,rgb color={255,169,0}] (-9.00,-2.00) -- (15.0,-0.400);
\draw[->,rgb color={255,152,0}] (-3.00,-1.00) -- (15.0,0.500);
\draw[->,rgb color={255,132,0}] (0.000,1.00) -- (15.0,3.50);
\draw[->,rgb color={255,152,0}] (-3.00,2.00) -- (10.5,6.50);
\draw[->,rgb color={255,169,0}] (-9.00,3.00) -- (-9.00,6.50);
\draw[->,rgb color={255,152,0}] (0.000,0.000) -- (8.25,-5.50);
\draw[->,rgb color={255,169,0}] (-3.00,-1.00) -- (-9.75,-5.50);
\draw[->,rgb color={255,152,0}] (0.000,1.00) -- (15.0,-2.33);
\draw[->,rgb color={255,169,0}] (-3.00,2.00) -- (15.0,-0.400);
\draw[->,rgb color={255,132,0}] (0.000,0.000) -- (15.0,-2.50);
\draw[->,rgb color={255,169,0}] (-9.00,-2.00) -- (-9.00,-5.50);
\draw[->,rgb color={255,152,0}] (-3.00,-1.00) -- (10.5,-5.50);
\draw[->,rgb color={255,132,0}] (0.000,1.00) -- (15.0,-0.667);
\draw[->,rgb color={255,152,0}] (-3.00,2.00) -- (15.0,0.500);
\draw[->,rgb color={255,169,0}] (-9.00,3.00) -- (15.0,1.40);
\draw[->,rgb color={255,169,0}] (1.80,0.400) -- (15.0,1.50);
\draw[->,rgb color={255,169,0}] (1.80,0.600) -- (15.0,-0.500);
\draw[->,green] (2.25,0.250) -- (15.0,0.250);
\draw[->,green] (-5.25,-1.75) -- (15.0,-1.75);
\draw[->,green] (0.000,-0.750) -- (15.0,-0.750);
\draw[->,green] (1.50,1.25) -- (15.0,1.25);
\draw[->,green] (-2.25,2.25) -- (15.0,2.25);
\draw[->,green] (2.25,0.750) -- (15.0,0.750);
\draw[->,green] (-2.25,-1.25) -- (15.0,-1.25);
\draw[->,green] (1.50,-0.250) -- (15.0,-0.250);
\draw[->,green] (0.000,1.75) -- (15.0,1.75);
\draw[->,green] (-5.25,2.75) -- (15.0,2.75);
\draw[->,brown] (1.80,0.400) -- (15.0,0.400);
\draw[->,brown] (0.000,-0.600) -- (15.0,-0.600);
\draw[->,brown] (0.600,1.40) -- (15.0,1.40);
\draw[->,brown] (1.80,0.600) -- (15.0,0.600);
\draw[->,brown] (0.600,-0.400) -- (15.0,-0.400);
\draw[->,brown] (0.000,1.60) -- (15.0,1.60);
\draw[->,rgb color={255,152,0}] (0.000,0.000) -- (15.0,1.25);
\draw[->,rgb color={255,169,0}] (-3.00,-1.00) -- (15.0,0.200);
\draw[->,rgb color={255,152,0}] (0.000,1.00) -- (15.0,2.67);
\draw[->,rgb color={255,169,0}] (-3.00,2.00) -- (15.0,5.00);
\draw[->,rgb color={255,152,0}] (0.000,0.000) -- (15.0,-1.67);
\draw[->,rgb color={255,169,0}] (-3.00,-1.00) -- (15.0,-4.00);
\draw[->,rgb color={255,152,0}] (0.000,1.00) -- (15.0,-0.250);
\draw[->,rgb color={255,169,0}] (-3.00,2.00) -- (15.0,0.800);
\draw[->,rgb color={255,169,0}] (0.000,0.000) -- (15.0,1.00);
\draw[->,rgb color={255,169,0}] (0.000,1.00) -- (15.0,2.25);
\draw[->,rgb color={255,169,0}] (0.000,0.000) -- (15.0,-1.25);
\draw[->,rgb color={255,169,0}] (0.000,1.00) -- (15.0,0.000);
\end{tikzpicture}
\caption{The scattering diagram $\mathscr{S}_\infty(\mathbb{P}^2,E)$ to $y$-order $6\cdot 3=18$.}
\label{fig:scatP2}
\end{figure}
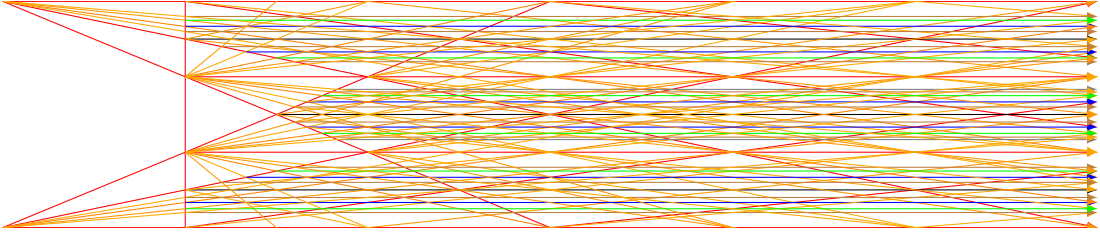

Note that in Theorem \ref{thm:main} we consider the $y$-order of rays instead of the $t$-order. More correctly we would need to consider the scattering diagrams $\mathscr{S}_k(X^{tor})$ with respect to some piecewise linear function $\varphi : \mathbb{R}^2 \rightarrow \mathbb{R}$ and the $t$-order of a ray $\mathfrak{d}$ would be $\varphi(m_{\mathfrak{d}})$. Then the $t$-order would indeed agree with the $y$-order. But this would make the notation and calculation unnecessarily complicated, so we don't do this here and just consider the $y$-order.

Note that using Theorem \ref{thm:main} we only get $R_d(X,D):=\sum_{\beta : D.\beta=d} R_\beta(X,D)$. This is no problem for $(\mathbb{P}^2,E)$, since classes $\beta$ of $\mathbb{P}^2$ are determined by their degree and thus their intersection with $E$. But for other cases we lose some information. To keep track of $\beta$ we need to consider smooth toric models of $X$.

\subsection{Classes via smooth toric models}
\label{S:classes}

Let $(X,D)$ be a very ample log Calabi-Yau pair and let $X^{tor}$ be a toric model of $X$ with fan $\Sigma$. One can find a refinement $\widetilde{\Sigma}$ of $\Sigma$ such that the generators of each two neighboring rays form a basis of $\mathbb{Z}^2$ and such that the spanning polytope of $\widetilde{\Sigma}$ equals $\Delta$.. Let $\widetilde{X}^{tor}$ the toric variety defined by $\widetilde{\Sigma}$. Then $\widetilde{X}^{tor}$ is smooth and $-K_{\widetilde{X}^{tor}}$ is nef and big, but it is ample if and only if $\widetilde{X}^{tor}=X^{tor}$. In other words, $\widetilde{X}^{tor}$ is a smooth toric weak Fano surface.

Now there is an isomorphism $H_2^+(X,\mathbb{Z}) \simeq H_2^+(\widetilde{X}^{tor},\mathbb{Z})$ between the group of effective curve classes on $X$ and the group effective curve classes on $\widetilde{X}^{tor}$. There is an effective curve class of $\widetilde{X}^{tor}$ associated to any ray of $\widetilde{\Sigma}$, hence to any integral point on the boundary of $\Delta(\widetilde{X}^{tor})=\Delta(X^{tor})$ and of $\Delta_\infty(\widetilde{X}^{tor})=\Delta_\infty(X^{tor})$.

Subdivide the edges $e$ of $\Delta_\infty(\widetilde{X}^{tor})$ into line segments $e'$ of affine length $1$. Let $\mathscr{S}_0(\widetilde{X}^{tor})$ be the scattering diagram that has for each $e'$ two rays $\mathfrak{d}_{e'}^\pm = p_{e'} + \mathbb{R}_{\geq 0}m_{e'}^\pm$ with functions $f_{\mathfrak{d}_{e'}^\pm}=1+tz^{m_{e'}^\pm}$, for $p_{e'}$ some interior point and $m_{e'}^\pm\in\mathbb{Z}^2$ as in \S\ref{S:initial}. This is related to $\mathscr{S}_0(X^{tor})$ via deformation of scattering diagrams as in \cite{GPS}, {\S}1.4.

For a vertex $v$ of $\Delta_\infty(\widetilde{X}^{tor})$ with coordinates $(x,y)$ let $\beta_v$ be the associated effective curve class of $X$ and let $L_v$ be the vertical line with $x$-coordinate $x+\epsilon$, where $\epsilon\in\mathbb{R}$ is sufficiently small (we need $|\epsilon|<1/d$ if we want to count curves with $D.\beta\leq d$).

A ray $\mathfrak{d}$ of $\mathscr{S}_\infty(\widetilde{X}^{tor})$ can be completed to a tropical curve $h_{\mathfrak{d}}$ by adding edges for its ancestors (the rays that are needed to obtain $\mathfrak{d}$ from $\mathscr{S}_0(\widetilde{X}^{tor})$), see \S\ref{S:scattrop}.

Let $d_v$ be the intersection of $h_{\mathfrak{d}}$ with $L_v$, where points of intersection are counted with multiplicity $\textup{mult}_p(h_{\mathfrak{d}},L_v)=|\braket{m^\bot,m'}|=|\textup{det}(m|m')|$. Here $m$ and $m'$ are the primitive integral tangent vectors at $p$ of $h_{\mathfrak{d}}$ and $L_v$, respectively. Note that $m'=(0,1)$, so $\textup{mult}_p(h_{\mathfrak{d}},L_v)=|m_1|$. This is an example of the tropical intersection pairing of \cite{Rud}. Define $\beta_{\mathfrak{d}}=\sum_v d_v\beta_v$, where the sum is over vertices of $\Delta_\infty(X^{tor})$. 

For a primitive curve class $\beta$ define $f_\beta = \prod_{\mathfrak{d} : m_{\mathfrak{d}}=(0,1), \beta_{\mathfrak{d}}=k\beta} f_{\mathfrak{d}}|_{t=1}$. Then we have
\[ \textup{log }f_\beta = \sum_{k=1}^\infty k(D.\beta) R_{k\beta}(X,D) y^{k(D.\beta)}. \]

\begin{expl}
\label{expl:8'a2}
Consider case (8'a) from Figure \ref{fig:16}: $X=\mathbb{P}^1\times\mathbb{P}^1$ and $X^{tor}=\mathbb{P}(1,1,2)$. The refinement $\widetilde{\Sigma}$ adds one ray and we have $\widetilde{X}^{tor}=\mathbb{F}_2$, the second Hirzebruch surface. This is the $\mathbb{P}^1$-bundle over $\mathbb{P}^1$ associated to the sheaf $\mathcal{O}_{\mathbb{P}^1}\oplus\mathcal{O}_{\mathbb{P}^1}(-2)$. Note that $\mathbb{F}_2$ is smooth but not Fano. The group of effective curve classes of $\mathbb{F}_2$ is generated by the class of a fiber $F$ and the class of a section, e.g. the exceptional divisor $C$ of the blow up $\mathbb{F}_2\rightarrow\mathbb{P}(1,1,2)$. The intersection numbers are $F^2=0$, $C^2=-2$ and $C\cdot F=1$. The group of effective curve classes of $\mathbb{P}^1\times\mathbb{P}^1$ is generated by its two rulings $L_1$ and $L_2$, with $L_1^2=0$, $L_2^2=0$ and $L_1\cdot L_2=1$. There is an isomorphism
\[ H_2^+(\mathbb{F}_2,\mathbb{Z}) \iso H_2^+(\mathbb{P}^1\times\mathbb{P}^1,\mathbb{Z}), \ F \mapsto L_2, \ E \mapsto L_1-L_2. \]
\end{expl}

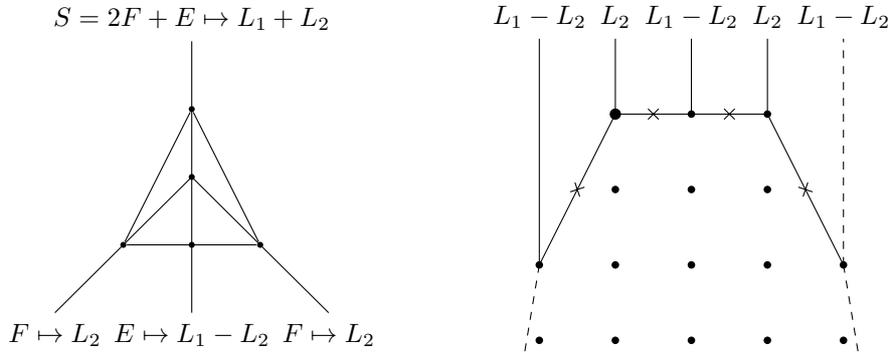
\begin{figure}[h!]
\centering
\begin{tikzpicture}[scale=0.9]
\coordinate[fill,circle,inner sep=0.8pt] (0) at (0,0);
\coordinate[fill,circle,inner sep=0.8pt] (1) at (-1,-1);
\coordinate[fill,circle,inner sep=0.8pt] (2) at (1,-1);
\coordinate[fill,circle,inner sep=0.8pt] (3) at (0,1);
\coordinate (1a) at (-2,-2);
\coordinate (2a) at (2,-2);
\coordinate (3a) at (0,2);
\coordinate (12) at (-0.5,-1);
\coordinate (12a) at (-1.5,-2);
\coordinate (12b) at (-0.5,-2);
\coordinate (12') at (0.5,-1);
\coordinate (12'a) at (0.5,-2);
\coordinate (12'b) at (1.5,-2);
\coordinate (23) at (0.5,0);
\coordinate (23a) at (0.5,2);
\coordinate (23b) at (2,-1.5);
\coordinate (31) at (-0.5,0);
\coordinate (31a) at (-0.5,2);
\coordinate (31b) at (-2,-1.5);
\draw (1) -- (2) -- (3) -- (1);
\draw (0,0) -- (1a) node[below]{\footnotesize$F\mapsto L_2$};
\draw (0,0) -- (2a) node[below]{\footnotesize$F\mapsto L_2$};
\draw (0,0) -- (3a) node[above]{\footnotesize$S=2F+E\mapsto L_1+L_2$};
\draw (0,0) -- (0,-2) node[below]{\footnotesize$E\mapsto L_1-L_2$};
\coordinate[fill,circle,inner sep=0.8pt] (b) at (0,-1);
\end{tikzpicture}
\hspace{1cm}
\begin{tikzpicture}[scale=1]
\draw (-1,-2) -- (0,0) -- (2,0) -- (3,-2);
\draw[dashed] (3,-2) -- (3.2,-3.2);
\draw[dashed] (-1,-2) -- (-1.2,-3.2);
\draw (-1,-2) -- (-1,1) node[above]{\footnotesize$L_1-L_2$};
\draw (0,0) -- (0,1) node[above]{\footnotesize$L_2$};
\draw (1,0) -- (1,1) node[above]{\footnotesize$L_1-L_2$};
\draw (2,0) -- (2,1) node[above]{\footnotesize$L_2$};
\draw[dashed] (3,-2) -- (3,1) node[above]{\footnotesize$L_1-L_2$};
\draw (0,0) node[fill,circle,inner sep=1.5pt]{};
\draw (0,-1) node[fill,circle,inner sep=1pt]{};
\draw (0,-2) node[fill,circle,inner sep=1pt]{};
\draw (0,-3) node[fill,circle,inner sep=1pt]{};
\draw (1,0) node[fill,circle,inner sep=1pt]{};
\draw (1,-1) node[fill,circle,inner sep=1pt]{};
\draw (1,-2) node[fill,circle,inner sep=1pt]{};
\draw (1,-3) node[fill,circle,inner sep=1pt]{};
\draw (2,0) node[fill,circle,inner sep=1pt]{};
\draw (2,-1) node[fill,circle,inner sep=1pt]{};
\draw (2,-2) node[fill,circle,inner sep=1pt]{};
\draw (2,-3) node[fill,circle,inner sep=1pt]{};
\draw (-1,-2) node[fill,circle,inner sep=1pt]{};
\draw (-1,-3) node[fill,circle,inner sep=1pt]{};
\draw (3,-2) node[fill,circle,inner sep=1pt]{};
\draw (3,-3) node[fill,circle,inner sep=1pt]{};
\draw (0.5,0) node[fill,cross,inner sep=2pt]{};
\draw (1.5,0) node[fill,cross,inner sep=2pt]{};
\draw (-0.5,-1) node[fill,cross,inner sep=2pt,rotate=63.435]{};
\draw (2.5,-1) node[fill,cross,inner sep=2pt,rotate=-63.435]{};
\end{tikzpicture}
\caption{Spanning polytope (left) and its unfolding (right) of the smooth toric model $\widetilde{X}^{tor}=\mathbb{F}_2$ of $X=\mathbb{P}^1\times\mathbb{P}^1$. The classes associated to rays of $\widetilde{\Sigma}$ resp. integral boundary points of $\Delta_\infty(\widetilde{X}^{tor})$ are shown.}
\label{fig:class}
\end{figure}

\section{Explanation via tropical curves}

The reason behind Theorem \ref{thm:main} is a general principle called \textit{tropical correspondence}, which relates counts of algebraic curves with counts of piecewise linear objects called \textit{tropical curves}. The rays of the scattering diagram $\mathscr{S}_\infty(X^{tor})$ correspond to tropical curves on a particular affine manifold with singularities $B$, the \textit{tropicalization} of $(X,D)$, and $\textup{log }f_{\textup{out}}$ is a generating function for counts of tropical curves. Together with tropical correspondence this gives a proof of Theorem \ref{thm:main}. In this section I will briefly sketch this correspondence. More details can be found in \cite{Gra}.

\subsection{Tropicalization of the toric model}
\label{S:toric}

Let $(X,D)$ be a very ample log Calabi-Yau pair and let $X^{tor}$ be a toric model of $X$. Let $\mathscr{P}$ be the polyhedral subdivision (subdivision into polytopes) of $\mathbb{R}^2$ whose $1$-dimensional cells are the edges of $\Delta(X^{tor})$ and $\rho\setminus\textup{Int }\Delta(X^{tor})$ for $\rho$ the rays of $\Sigma$. We call $(\mathbb{R}^2,\mathscr{P})$ the \textit{tropicalization} or \textit{dual intersection complex} of $(X^{tor},\partial X^{tor})$, where $\partial X^{tor}$ is the toric boundary. 

There is a degeneration $(\mathcal{X},\mathcal{D})$ of $(X^{tor},\partial X^{tor})$ into toric pieces (\textit{toric degeneration}) such that vertices $v$ of $\mathscr{P}$ correspond to the irreducible components $X_v$ of the central fiber of $\mathcal{X}$. The fan of $X_v$ is given by $\mathscr{P}$ locally around $v$. If we consider $\mathcal{X}$ with the divisorial log structure by $X_0\cup\mathcal{D}$, then $(B,\mathscr{P})$ is the tropicalization of the log scheme $\mathcal{X}$ in the sense of \cite{GSloggw}, Appendix B. This justifies the name.

\begin{expl}
The tropicalization of $(\mathbb{P}^2,\partial\mathbb{P}^2)$ is shown on the left hand side of Figure \ref{fig:dualint}. A toric degeneration of $(\mathbb{P}^2,\partial\mathbb{P}^2)$ is given by
\begin{eqnarray*}
\mathcal{X} &=& \{XYZ=tW\} \subset \mathbb{P}(1,1,1,3)\times\mathbb{A}^1 \rightarrow \mathbb{A}^1 \\
\mathcal{D} &=& \{W=0\} \subset \mathcal{X}
\end{eqnarray*}
where $X,Y,Z,W$ are the coordinates of $\mathbb{P}(1,1,1,3)$ and the map $\mathcal{X}\rightarrow\mathbb{A}^1$ is given by projection to $t$. The central fiber of $\mathcal{X}$ consists of three $\mathbb{P}(1,1,3)$ glued along toric divisors, according to the combinatorics of $\mathscr{P}$.
\end{expl}

\subsection{Tropicalization of $(X,D)$}
\label{S:XD}

Let $(X,D)$ be a very ample log Calabi-Yau pair and let $X^{tor}$ be a toric model of $X$. Let $(\mathbb{R}^2,\mathscr{P})$ be the tropicalization of $(X^{tor},\partial X^{tor})$. For each edge $e$ of $\mathscr{P}$ choose a point $p_e$ in the interior (e.g. the middle point) and cut out a wedge from $\mathbb{R}^2$ whose boundary consists of two rays with base $p_e$ and directions given by the directions of the unbounded edges of $\mathscr{P}$ adjacent to $e$. Identify the rays $\rho_1,\rho_2$ of the wedge via the unique $SL(2,\mathbb{Z})$-transformation that maps $\rho_1$ to $\rho_2$ and leaves $e$ invariant. Doing this for all edges $e$ we obtain an affine manifold with singularities $B$. This means away from some singular points $B$ is a topological manifold whose transition functions are affine transformations. The singularities are the points $p_e$ and the monodromy around $p_e$ is conjugate to $\left(\begin{smallmatrix}1 & l_e \\ 0 & 1 \end{smallmatrix}\right)$, where $l_e$ is the affine length of $e$. This is a consequence of $X$ being Fano. We call $(B,\mathscr{P})$ the \textit{tropicalization} of $(X,D)$. 

Again we have a toric degeneration $(\mathcal{X},\mathcal{D})$ of $(X,D)$ whose central fiber describes $(B,\mathscr{P})$ as in \S\ref{S:toric}. It can be obtained from the toric degeneration of $(X^{tor},\partial X^{tor})$ via deformation.

\begin{expl}
The tropicalization of $(\mathbb{P}^2,E)$ is shown on the right hand side of Figure \ref{fig:dualint}. The shaded regions are cut out and the dashed lines are mutually identified, leading to three affine singularities, pictured by crosses, with monodromy conjugate to $\left(\begin{smallmatrix}1 & 1 \\ 0 & 1 \end{smallmatrix}\right)$. These affine transformation have to be applied when passing across a shaded region. One of them is shown in Figure \ref{fig:dualint}. Note that the three unbounded edges are all parallel and $B$ has only one unbounded direction. 

A toric degeneration of $(\mathbb{P}^2,E)$ is given by
\begin{eqnarray*}
\mathcal{X} &=& \{XYZ=t(W+f_3)\} \subset \mathbb{P}(1,1,1,3)\times\mathbb{A}^1 \rightarrow \mathbb{A}^1 \\
\mathcal{D} &=& \{W=0\} \subset \mathcal{X}
\end{eqnarray*}
where $f_3$ is a general degree $3$ polynomial in $X,Y,Z$. Smoothness of $D$ corresponds to the fact that $B$ has only one unbounded direction.
\end{expl}

\begin{figure}[h!]
\centering
\begin{tikzpicture}
\coordinate[fill,circle,inner sep=1pt] (1) at (1,0);
\coordinate[fill,circle,inner sep=1pt] (2) at (0,1);
\coordinate[fill,circle,inner sep=1pt] (3) at (-1,-1);
\coordinate (1a) at (2.5,0);
\coordinate (2a) at (0,2.5);
\coordinate (3a) at (-2.5,-2.5);
\draw (1) -- (2) -- (3) -- (1);
\draw (1) -- (1a);
\draw (2) -- (2a);
\draw (3) -- (3a);
\end{tikzpicture}
\hspace{2cm}
\begin{tikzpicture}
\coordinate[fill,circle,inner sep=1pt] (1) at (1,0);
\coordinate[fill,circle,inner sep=1pt] (2) at (0,1);
\coordinate[fill,circle,inner sep=1pt] (3) at (-1,-1);
\coordinate (1a) at (2.5,0);
\coordinate (2a) at (0,2.5);
\coordinate (3a) at (-2.5,-2.5);
\coordinate (12) at (0.5,0.5);
\coordinate (12a) at (0.5,2.5);
\coordinate (12b) at (2.5,0.5);
\coordinate (23) at (-0.5,0);
\coordinate (23a) at (-0.5,2.5);
\coordinate (23b) at (-2.5,-2);
\coordinate (31) at (0,-0.5);
\coordinate (31a) at (2.5,-0.5);
\coordinate (31b) at (-2,-2.5);
\draw (1) -- (2) -- (3) -- (1);
\draw (1) -- (1a);
\draw (2) -- (2a);
\draw (3) -- (3a);
\fill[opacity=0.2] (12) -- (12a) -- (12b) -- (12);
\fill[opacity=0.2] (23) -- (23a) -- (23b) -- (23);
\fill[opacity=0.2] (31) -- (31a) -- (31b) -- (31);
\draw[dashed] (12a) -- (12) -- (12b);
\draw[dashed] (23a) -- (23) -- (23b);
\draw[dashed] (31a) -- (31) -- (31b);
\draw (12) node[fill,cross,inner sep=2pt,rotate=45]{};
\draw (23) node[fill,cross,inner sep=2pt,rotate=63.43]{};
\draw (31) node[fill,cross,inner sep=2pt,rotate=26.57]{};
\draw[->] (1.5,0.2) to[bend right=45] (0.2,1.5);
\draw (1.7,1.6) node{\scriptsize$\begin{pmatrix}0 & -1 \\ 1 & 2 \end{pmatrix}$};
\end{tikzpicture}
\caption{The tropicalization of $(\mathbb{P}^2,\partial\mathbb{P}^2)$ (left) and $(\mathbb{P}^2,E)$ (right).}
\label{fig:dualint}
\end{figure}
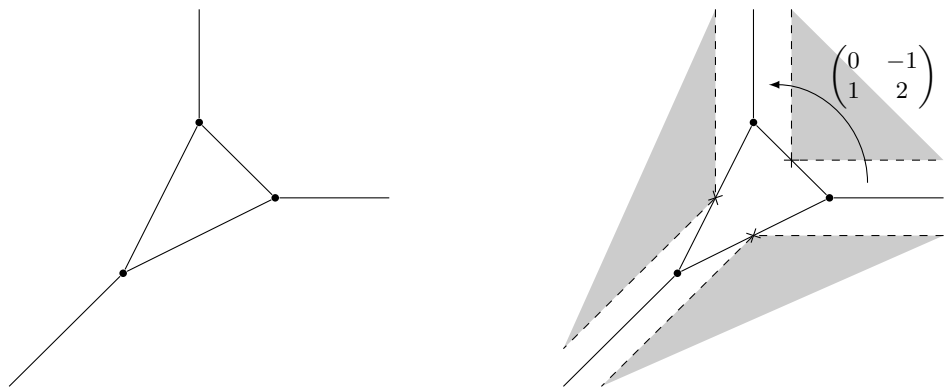

\subsection{Unfolding: chart at infinity and discrete covering space}
\label{S:unfold2}

The right hand side of Figure \ref{fig:dualint} shows the affine manifold with singularities $B$ in a chart on the bounded maximal cell. If we instead consider $B$ in a chart on an unbounded maximal cell, we get the picture on the left hand side of Figure \ref{fig:cover}. You can image going from the former to the latter by cutting out the shaded regions and mutually gluing the dashed lines. To still get an embedding into $\mathbb{R}^2$ we have to cut out some part of the bounded maximal cell, and we have to cut along some unbounded edge.

From Figure \ref{fig:cover} we see that a discrete covering space of $(B,\mathscr{P})$ contains the unfolding of the spanning polytope $\Delta_\infty(X^{tor})$ we defined in \S\ref{S:unfold}. There is one affine singularity on the interior of each edge e and it has monodromy $\left(\begin{smallmatrix}1 & l_e \\ 0 & 1 \end{smallmatrix}\right)$. By \cite{GSreconstruction} this gives an initial scattering diagram precisely as defined in \S\ref{S:initial}

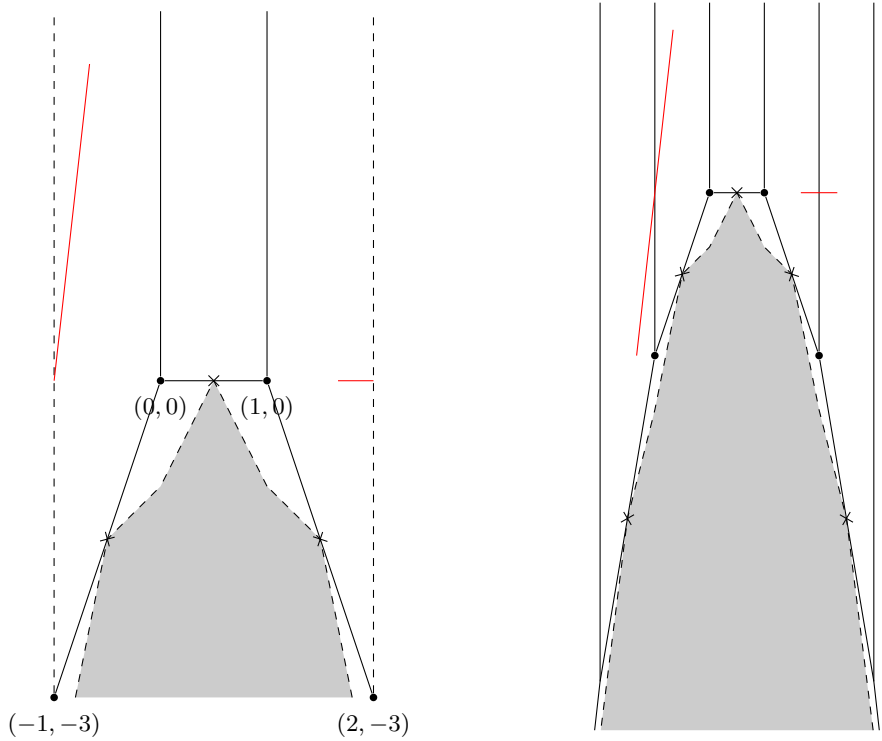
\begin{figure}[h!]
\centering
\begin{tikzpicture}[scale=1.4,rotate=90]
\coordinate[label=below:\scriptsize${(1,0)}$,fill,circle,inner sep = 1pt] (1) at (0,0);
\coordinate[label=below:\scriptsize${(0,0)}$,fill,circle,inner sep = 1pt] (2) at (0,1);
\coordinate[label=below:\scriptsize${(2,-3)}$,fill,circle,inner sep = 1pt] (3) at (-3,-1);
\coordinate[label=below:\scriptsize${(-1,-3)}$,fill,circle,inner sep = 1pt] (4) at (-3,2);
\draw (3) -- (1) -- (2) -- (4);
\draw (1) -- (3.5,0);
\draw (2) -- (3.5,1);
\draw[dashed] (3) -- (3.5,-1);
\draw[dashed] (4) -- (3.5,2);
\draw[dashed] (-1,0) -- (0,0.5) -- (-1,1);
\draw[dashed] (-3,-0.8) -- (-1.5,-0.5) -- (-1,0);
\draw[dashed] (-3,1.8) -- (-1.5,1.5) -- (-1,1);
\fill[opacity=0.2] (-3,1.8) -- (-1.5,1.5) -- (-1,1) -- (0,0.5) -- (-1,0) -- (-1.5,-0.5) -- (-3,-0.8);
\coordinate[fill,cross,inner sep=2pt] (5) at (0,0.5);
\coordinate[fill,cross,inner sep=2pt,rotate=26.57] (6) at (-1.5,-0.5);
\coordinate[fill,cross,inner sep=2pt,rotate=-30] (7) at (-1.5,1.5);
\draw[red] (0,2) -- (3,1.667);
\draw[red] (0,-1) -- (0,-0.667);
\end{tikzpicture}
\hspace{2cm}
\begin{tikzpicture}[scale=0.72,rotate=90]
\coordinate[fill,circle,inner sep = 1pt] (1) at (0,0);
\coordinate[fill,circle,inner sep = 1pt] (2) at (0,1);
\coordinate[fill,circle,inner sep = 1pt] (3) at (-3,-1);
\coordinate[fill,circle,inner sep = 1pt] (4) at (-3,2);
\draw (3) -- (1) -- (2) -- (4);
\draw (1) -- (3.5,0);
\draw (2) -- (3.5,1);
\draw (3) -- (3.5,-1);
\draw (4) -- (3.5,2);
\draw (3) -- (-9,-2) -- (-9.9,-2.1);
\draw (4) -- (-9,3) -- (-9.9,3.1);
\draw (-9,-2) -- (3.5,-2);
\draw (-9,3) -- (3.5,3);
\draw[dashed] (-9.9,-1.9875) -- (-6,-1.5) -- (-4,-1) -- (-1.5,-0.5) -- (-1,0) -- (0,0.5) -- (-1,1) -- (-1.5,1.5) -- (-4,2) -- (-6,2.5) -- (-9.9,2.9875);
\fill[opacity=0.2] (-9.9,-1.9875) -- (-6,-1.5) -- (-4,-1) -- (-1.5,-0.5) -- (-1,0) -- (0,0.5) -- (-1,1) -- (-1.5,1.5) -- (-4,2) -- (-6,2.5) -- (-9.9,2.9875);
\coordinate[fill,cross,inner sep=2pt] (5) at (0,0.5);
\coordinate[fill,cross,inner sep=2pt,rotate=26.57] (6) at (-1.5,-0.5);
\coordinate[fill,cross,inner sep=2pt,rotate=-30] (7) at (-1.5,1.5);
\coordinate[fill,cross,inner sep=2pt,rotate=15] (8) at (-6,-1.5);
\coordinate[fill,cross,inner sep=2pt,rotate=-15] (9) at (-6,2.5);
\draw[red] (3,1.667) -- (-3,2.333);
\draw[red] (0,-0.667) -- (0,-1.333);
\draw (-10.2,0);
\end{tikzpicture}
\caption{The tropicalization $(B,\mathscr{P})$ of $(\mathbb{P}^2,E)$ in a chart on an unbounded cell (left) and a discrete covering space $\bar{B}\rightarrow B$ of it (right). A straight line on $B$ and its preimage in $\bar{B}$ are shown in red.}
\label{fig:cover}
\end{figure}

\subsection{Tropical curves}

Let $B$ be an affine manifold with singularities.

\begin{defi}
A tropical curve $h : \Gamma \rightarrow B$ is a piecewise affine map from a weighted graph $\Gamma$ without $1$- and $2$-valent vertices, possibly with some non-compact edges (legs), such that (images of) legs are either unbounded or end in affine singularities of $B$, and such that at each vertex $V$ of $\Gamma$ the balancing condition holds:
\[ \sum_{E\ni V} w_Em_{(V,E)}. \]
Here $w_E$ is the weight of $E$ and $m_{(V,E)}$ is the primitive integral tangent vector pointing from $V$ to $E$.
\end{defi}

Note that locally near an affine singularity $\delta$ of $B$ the edges and legs of a tropical curve $h:\Gamma\rightarrow B$ must have images whose directions are invariant under the monodromy transformation around $\delta$. When $(B,\mathscr{P})$ is a tropicalization of a pair $(X,D)$ as in \S\ref{S:XD} this means that they are parallel to the edge of $\mathscr{P}$ containing $\delta$.

\begin{defi}
\label{defi:mult}
Let $h : \Gamma \rightarrow B$ be a tropical curve. For a trivalent vertex $V\in V(\Gamma)$ define
\[ m_V=\lvert u_{(V,E_1)}\wedge u_{(V,E_2)}\rvert=\lvert\textup{det}(u_{(V,E_1)}|u_{(V,E_2)})\rvert, \]
where $E_1,E_2$ are any two edges adjacent to $V$. For a vertex $V\in V(\Gamma)$ of valency $\nu_V>3$ let $h_V$ be the one-vertex tropical curve describing $h$ locally at $V$ and let $h'_V$ be a deformation of $h_V$ to a trivalent tropical curve. 
This deformation has $\nu_V-2$ vertices. We define 
\[ m_V=\prod_{V'}m_{V'} \]
where the product is over all vertices $V'$ of $h'_V$. By Proposition 2.7 in \cite{GPS}, this expression is independent of the deformation $h'_V$, hence well-defined. For a bounded leg $L$ with weight $w_L$ define 
\[ m_L=\frac{(-1)^{w_L+1}}{w_L^2}. \]
We define the \textit{multiplicity} of $h$ to be
\[ m_h = \frac{1}{|\textup{Aut}(h)|} \cdot \prod_V m_V \cdot \prod_L m_E, \]
where the first product is over all vertices of $\Gamma$ and the second product is over all legs $L$ of $\Gamma$ ending in affine singularities of $B$.
\end{defi}

\subsection{Tropical correspondence}

The tropical correspondence theorem for very ample log Calabi-Yau pairs $(X,D)$ is the following.

\begin{thm}[\cite{Gra}, Theorem 1]
\label{thm:trop}
\[ R_\beta(X,D) = \sum_{h \in \mathcal{H}_\beta(X,D)} \textup{Mult}(h). \]
\end{thm}

\begin{proof}[Idea of proof]
Consider a curve (stable log map) to $(X,D)$. Its limit under the toric degeneration $(\mathcal{X},\mathcal{D})$ has several irreducible components, each mapping to a particular component of $X_0$. The tropical curves $h:\Gamma\rightarrow B$ on the tropicalization $(B,\mathscr{P})$ describe the combinatorics of these limiting stable log maps to $X_0$, and the multiplicity $\textup{Mult}(h)$ is exactly the number of lifts from a stable log map to $X_0$ to a stable log map to $X$. This is made precise by the \textit{degeneration formula} \cite{KLR}. Thus Theorem \ref{thm:trop} is basically a consequence of the degeneration formula. But there is one subtlety here. The degeneration formula only works for log smooth families, but $\mathcal{X}$ has log singularities corresponding to the affine singularities of $B$. As described in \cite{Gra}, by blowing up $\mathcal{X}$ along irreducible components of $X_0$ we obtain a log smooth degeneration $\widetilde{\mathcal{X}}$ of $X$. The exceptional locus of the blowup $\widetilde{\mathcal{X}}\rightarrow \mathcal{X}$ is a disjoint union of projective lines, one for each log singularity resp. affine singularity. A stable log map to the central fiber $\widetilde{X}_0$ of $\widetilde{\mathcal{X}}$ might have components intersecting such an exceptional line or entirely mapping onto it. This leads to the formula for $m_L$ in Definition \ref{defi:mult}, which is exactly the log Gromov-Witten invariant of $w_L$-fold covers of an exceptional line totally ramified at a single point. Figure \ref{fig:pf} shows for $(\mathbb{P}^2,E)$: an irreducible nodal cubic, its limiting map to $\widetilde{X}_0$ and $X_0$, and its corresponding tropical curve on the tropicalization $(B,\mathscr{P})$.
\end{proof}

\begin{figure}[h!]
\centering
\begin{tikzpicture}[scale=1]
\draw[smooth,samples=100,domain=-0.5:0.75] plot (\x,{sqrt((2*\x)^3+1)/2+5});
\draw[smooth,samples=100,domain=-0.5:0.75] plot (\x,{-sqrt((2*\x)^3+1)/2+5});
\draw[blue,thick,smooth,samples=100,domain=-1:0.5] plot ({\x*sqrt(\x+1)/2},\x/2+6);
\draw[blue,thick,smooth,samples=100,domain=-1:0.5] plot ({-\x*sqrt(\x+1)/2},\x/2+6);
\draw[->] (0,3.5) -- (0,2.5);
\draw (0,3) node[right]{central fiber};
\coordinate (0) at (0,0);
\coordinate (1) at (-1,-1);
\coordinate (2) at (2,-1);
\coordinate (3) at (-1,2);
\coordinate[fill,circle,inner sep=0.8pt] (1a) at (-0.5,-0.5);
\coordinate[fill,circle,inner sep=0.8pt] (2a) at (1,-0.5);
\coordinate[fill,circle,inner sep=0.8pt] (3a) at (-0.5,1);
\draw (1) -- (2) -- (3) -- (1);
\draw (0) -- (1);
\draw (0) -- (2);
\draw (0) -- (3);
\draw (1a) -- (-0.6,0);
\draw (2a) -- (0.69,-0.11);
\draw (3a) -- (-0.11,0.69);
\draw[thick,blue] (-1,-1) to[bend right=10] (1,-0.5);
\draw[thick,blue] (1,-0.5) -- (0.69,-0.11);
\draw[thick,blue] (-1,-1) to[bend right=30] (-0.5,1) to[bend right=20] (-0.7,1.2);
\draw[thick,blue] (-0.5,1) -- (-0.11,0.69);
\draw[thick,blue] (-1,-1) to[bend right=10] (-1.4,-1.4);
\draw[blue] (-0.5,1) node[right]{$2$};
\draw[blue] (-1,-1) node[below]{$9$};
\draw[->] (2,0.5) -- (3,0.5);
\draw (2.5,0.5) node[above]{tropicalization};
\end{tikzpicture}
\begin{tikzpicture}[rotate=90]
\coordinate (1) at (0,0);
\coordinate (2) at (0,1);
\coordinate (3) at (-3,-1);
\coordinate (4) at (-3,2);
\draw (3) -- (1) -- (2) -- (4);
\draw (1) -- (5,0);
\draw (2) -- (5,1);
\draw[dashed] (3) -- (5,-1);
\draw[dashed] (4) -- (5,2);
\draw[dashed] (-1,0) -- (0,0.5) -- (-1,1);
\draw[dashed] (-3,-0.8) -- (-1.5,-0.5) -- (-1,0);
\draw[dashed] (-3,1.8) -- (-1.5,1.5) -- (-1,1);
\fill[opacity=0.2] (-3,1.8) -- (-1.5,1.5) -- (-1,1) -- (0,0.5) -- (-1,0) -- (-1.5,-0.5) -- (-3,-0.8);
\coordinate[fill,cross,inner sep=2pt] (5) at (0,0.5);
\coordinate[fill,cross,inner sep=2pt,rotate=26.57] (6) at (-1.5,-0.5);
\coordinate[fill,cross,inner sep=2pt,rotate=-30] (7) at (-1.5,1.5);
\fill[color=blue] (0,0) circle (1.5pt);
\fill[color=blue] (0,1) circle (1.5pt);
\fill[color=blue] (2,0.66) circle (1.5pt);
\draw[thick,blue] (-1.5,1.5) -- (-0.75,1.25) node[left]{$2$} -- (0,1) -- (2,0.66) -- (0,0) -- (-1.5,-0.5);
\draw[thick,blue] (2,0.66) -- (3.5,0.66) node[right]{$9$} -- (5,0.66);
\draw[thick,blue] (0,1) -- (0,0.5);
\end{tikzpicture}
\caption{Tropical correspondence for an irreducible cubic for $(\mathbb{P}^2,E)$.}
\label{fig:pf}
\end{figure}
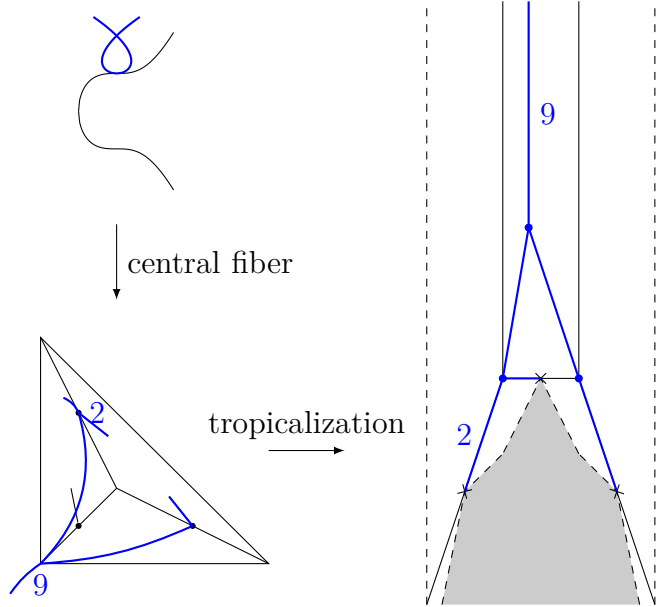

\subsection{Scattering and tropical curves}
\label{S:scattrop}

Let $\mathscr{S}_k$ be a scattering diagram consistent to some order $k$ and let $\mathfrak{d}$ be a ray in $\mathscr{S}_k$ with function $f_{\mathfrak{d}}=1+at^kz^{wm}$ for $w>0$ and $m\in\mathbb{Z}^2$ some primitive vector. We construct a tropical curve $h_{\mathfrak{d}}$ as follows. Let $\Gamma$ be the graph consisting of one vertex and one leg with weight $w$, and let $h : \Gamma \rightarrow B$ map this leg to $\mathfrak{d}$. This is certainly not balanced, so is no tropical curve. We recursively add new rays to $\Gamma$ to form a tropical curve.

Let $\mathscr{S}_{b_{\mathfrak{d}}}$ be $\mathscr{S}_k$ localized at the base $b_{\mathfrak{d}}$ of $\mathfrak{d}$ as described in \S\ref{S:scat}. Let $\mathfrak{d}_1,\ldots,\mathfrak{d}_r$ be the rays of $\mathscr{S}_{b_{\mathfrak{d}}}\setminus\{\mathfrak{d}\}$ and change their functions $f_{\mathfrak{d}_i}=1+a_it^{k_i}x^{m_i}$ to $f_{\mathfrak{d}_i}=1+s_ia_it^{k_i}x^{m_i}$ for finitely many new variables $s_i$. Now scattering gives $f_{\mathfrak{d}}=1+(\prod_i s_i^{w_i})at^kx^m$ and the collection $(w_i)_{i=1,\ldots,r}$ tells us which rays are involved in the scattering to produce $\mathfrak{d}$. If $w_i\neq 0$ and $\mathfrak{d}_i$ is an initial ray, then $b_{\mathfrak{d}_i}=p_e$ for some edge $e$ of $\mathscr{P}$. In this case we add a leg to $\Gamma$ with weight $w_i$ and mapped to the half-open line segment connecting $b_{\mathfrak{d}}$ with $p_e$. If $w_i\neq 0$ and $\mathfrak{d}_i$ is not an initial ray, we add a compact ($2$-valent) edge to $\Gamma$ with weight $w_i$ and image given by the line segment connecting $b_{\mathfrak{d}}$ with $b_{\mathfrak{d}_i}$. We repeat this procedure for all non-initial rays $\mathfrak{d}_i$ we found in this way. This procedure stops after finitely many steps and we obtain a tropical curve in $\mathfrak{H}_d$. The coefficient of $\mathfrak{d}$ equals $w\text{Mult}(h_{\mathfrak{d}})$. This gives the following:

\begin{prop}[\cite{Gra}, Proposition 5.20]
\label{prop:scattrop}
We have
\[ \textup{log }f_{\textup{out}} = \sum_{d=1}^\infty\sum_{h\in\mathcal{H}_d} 3d \cdot \textup{Mult}(h) \cdot t^{3d}y^{3d}. \]
\end{prop}

\begin{proof}
This follows from an inductive application of \cite{GPS}, Theorem 2.8.
\end{proof}

Proposition \ref{prop:scattrop} and Theorem \ref{thm:trop} together give Theorem \ref{thm:main}.

\section{Tutorial: scattering.sage}
\label{S:sage}

\subsection{Executing sage code}

SageMath is an open-source computer algebra system based on the python language. To execute sage code on your local system you need to install SageMath. Follow the installation instructions for your operating system as described on \href{https://www.sagemath.org/}{\texttt{https://www.sagemath.org}}. An installer for Windows can be downloaded at \href{https://github.com/sagemath/sage-windows/releases}{https://github.com/sagemath/sage-windows/releases}. Alternatively, you can execute sage code on the cloud computing service CoCalc. Go to \href{https://cocalc.com/}{\texttt{https://cocalc.com}}, create an account and create a new sage worksheet. According to my experience this is much slower, so I suggest the first option. If you opted for the first option and installed SageMath on your system, you can either use the web-based Jupyter notebook or open the SageMath Console (\texttt{SageMath x.x.exe} in Windows). In the latter you can use standard Unix and Shell commands \\
\begin{minipage}{\textwidth}
\vspace{5mm}
\hspace{1cm}
\begin{tabular}{m{6cm} l}
\texttt{cd <dir>}  		& change to directory \texttt{<dir>} \\
\texttt{cd ..}  			& go one directory back \\
\texttt{ls} 				& list files in current directory
\end{tabular}
\vspace{5mm}
\end{minipage}
If you navigated to the directory containing your sage file, you can execute it using \\
\begin{minipage}{\textwidth}
\vspace{5mm}
\hspace{1cm}
\begin{tabular}{m{6cm} l}
\texttt{load('scattering.sage')}   	& execute the file \texttt{scattering.sage}
\end{tabular}
\vspace{5mm}
\end{minipage}
Or you can load the code into the terminal and execute it then (press \texttt{Enter}) \\
\begin{minipage}{\textwidth}
\vspace{5mm}
\hspace{1cm}
\begin{tabular}{m{6cm} l}
\texttt{\%load scattering.sage}   	& load the contents of \texttt{scattering.sage} \\
\end{tabular}
\vspace{5mm}
\end{minipage}
The latter is slightly slower and it will litter your terminal with text, but it will give you better error messages for debugging.

\subsection{scattering.sage}

In the next sections I will describe how to compute scattering diagrams using the sage code \texttt{scattering.sage} that you can find on my webpage \href{https://timgraefnitz.com/}{https://timgraefnitz.com/}. If you download \texttt{scattering.sage} and execute it you won't see anything, but now all the methods implemented in \texttt{scattering.sage} are available for execution and you can follow the tutorial presented in the next sections. 

If you want to see quick success you can open \texttt{scattering.sage} with a simple text editor, scroll down to the very bottom of the code, uncomment one of the last two paragraphs, and execute the code again. The first paragraph will compute and display the scattering diagram for $(\mathbb{P}^2,E)$ up to degree $2$. The second will compute the standard scattering diagram with initial functions $(1+tx)^3$ and $(1+ty)^3$ as in \cite{GPS}, Example 1.6. Moreover, you can type \texttt{scattering()} for interactive usage. 

I am still improving the code continuously, so some functions might not behave exactly as described here.

\subsection{Initializing scattering diagrams}

Before you can initialize a scattering diagram, you need to initialize the ring $\mathbb{Q}[t_0,\ldots,t_{r-1}][x^{\pm 1},y^{\pm 1}]$ with enough $t$-variables suitable for your computation. \\
\begin{minipage}{\textwidth}
\vspace{5mm}
\hspace{1cm}
\begin{tabular}{m{6cm} l}
\texttt{initialize(r)}	 & initializes $\mathbb{Q}[t_0,\ldots,t_{r-1}][x^{\pm 1},y^{\pm 1}]$
\end{tabular}
\vspace{2mm}
\end{minipage}
Then you can create a scattering diagram by specifying it rays, for instance: \\
\begin{minipage}{\textwidth}
\vspace{5mm}
\hspace{1cm}
\begin{tabular}{m{6cm} l}
\texttt{D=Diagram([((0,0),(1,0),1+t{\textunderscore}0*x),((0,0),(0,1),1+t{\textunderscore}1*y)])} &
\end{tabular}
\vspace{5mm}
\end{minipage}
There are some predefined scattering diagrams that you can create easily:
\begin{minipage}{\textwidth}
\vspace{5mm}
\hspace{1cm}
\begin{tabular}{m{6cm} l l}
\texttt{D=Diagram('std',(m,n))} 	& $f_1=(1+t_1x)^m$ 	& $f_2=(1+t_2y)^n$ \\
\texttt{D=Diagram('exp',(m,n))} 	& $f_1=1+t_1x^m$ 	& $f_2=1+t_2y^n$ \\
\texttt{D=Diagram('det',m)} 		& $f_1=1+t_1x$ 		& $f_2=1+t_2x^{-1}y^m$
\end{tabular}
\vspace{5mm}
\end{minipage}

More important for us, you can create the scattering diagram of a very ample log Calabi-Yau pair $(X,D)$ by giving $X$ or the case of Figure \ref{fig:16} and the order $k$ you wish to compute. The order gives the number of fundamental domains (see \S\ref{S:unfold} and \S\ref{S:unfold2}) that are generated. This is necessary to compute the numbers $R_\beta(X,D)$ for $\textup{ord}(\beta)\leq k$ , where $\textup{ord}(\beta)=\textup{max}\{k \ |\ \exists \beta': \beta=k\beta'\}$.
\begin{minipage}{\textwidth}
\vspace{5mm}
\hspace{1cm}
\begin{tabular}{m{6cm} l}
\texttt{D=Diagram('P2',k)} 		& \texttt{D=Diagram('(9)',k)} \\
\texttt{D=Diagram('P1xP1',k)} 	& \texttt{D=Diagram('(8')',k)} \\
\texttt{D=Diagram('cubic',k)}	& \texttt{D=Diagram('(3a)',k)}
\end{tabular}
\vspace{5mm}
\end{minipage}
For spacing reasons the diagrams for $(X,D)$ are implemented with unbounded $x$-direction as in Figure \ref{fig:scatP2}, not unbounded $y$-direction as in the rest of this paper.

\subsection{Scattering computations}

Now we created a diagram, we can start scattering. To create from a diagram $\texttt{D}$ a diagram $\texttt{D1}$ consistent to $t$-order $k$ you write \\
\begin{minipage}{\textwidth}
\vspace{5mm}
\hspace{1cm}
\begin{tabular}{m{6cm} l}
\texttt{D1=D.scattering(k)} & gives the diagram of \texttt{D}  to order \texttt{k} \\
\end{tabular}
\vspace{5mm}
\end{minipage}
For diagrams of pairs $(X,D)$ this will be very slow. Already for $(\mathbb{P}^2,E)$ and degree $3$ it takes more than a day. But we can speed up the calculation by using the symmetry of the scattering diagram. To do so, give $X$ or the its label (as in Figure \ref{fig:16}) as a second argument to \texttt{scattering()}. For instance, \\
\begin{minipage}{\textwidth}
\vspace{5mm}
\hspace{1cm}
\begin{tabular}{m{6cm} l}
\texttt{D=Diagram('P2',3)} & \\
\texttt{D1=D.scattering(3,'P2')} &
\end{tabular}
\vspace{5mm}
\end{minipage}
will give the scattering diagram for $(\mathbb{P}^2,E)$ and degree $d\leq 3$ in less than a minute.

\subsection{Printing diagrams}
You have the following functions to display a diagram: \\
\begin{minipage}{\textwidth}
\vspace{5mm}
\hspace{1cm}
\begin{tabular}{m{6cm} l}
\texttt{print(D)} 	&	print \texttt{D} to the terminal \\
\texttt{D.draw()} 	&	display a \texttt{.png} file of \texttt{D} \\
\texttt{D.tex()} 		&	print tikz code for \texttt{D}
\end{tabular}
\vspace{5mm}
\end{minipage}
The functions \texttt{draw()} and \texttt{tex()} have the following optional arguments \\
\begin{minipage}{\textwidth}
\vspace{5mm}
\hspace{1cm}
\begin{tabular}{m{6cm} l}
\texttt{special:list(Ray)}   			& list of special rays that are printed fat \\
\texttt{colors:boolean} 			& color rays depending on their order \\
\texttt{functions:boolean} 			& show functions of rays \\
\texttt{directions:list(tuple)}	   	& color and print functions only for rays \\
								& with these directions
\end{tabular}
\end{minipage}
For instance, \\
\begin{minipage}{\textwidth}
\vspace{5mm}
\hspace{1cm}
\begin{tabular}{m{6cm} l}
\texttt{D=Diagram('P2',6).scattering(6,'P2')} & \\
\texttt{D.tex(colors=True,directions=[(1,0)])} &
\end{tabular}
\vspace{5mm}
\end{minipage}
gives the tikz code to produce Figure \ref{fig:scatP2}, up to some scaling. This will take about two hours on a fast computer. From \texttt{print(D1)} you can read off the function $f_{\textup{out}}$ used in Example \ref{expl:scat} to compute $R_1(\mathbb{P}^2,E)=9$, $R_2(\mathbb{P}^2,E)=\frac{135}{4}$ and $R_3(\mathbb{P}^2,E)=244$.

\subsection{Saving and loading diagrams}

You can save a diagram by \\
\begin{minipage}{\textwidth}
\vspace{5mm}
\hspace{1cm}
\begin{tabular}{m{6cm} l}
\texttt{D.save(<file path>)} 	& saves a ray list of \texttt{D} in \texttt{<file path>} \\
\texttt{D.save()}				& saves in \texttt{'calculations.txt'}
\end{tabular}
\vspace{5mm}
\end{minipage}
When you initialize a diagram of a pair $(X,D)$ the code will look for saved diagrams in \texttt{calculations.txt}. You can change the path where it looks by giving it as an argument \texttt{path} to \texttt{Diagram()}. If you call \texttt{scattering()} with an order that has already been calculated, the code will tell you so and not compute it again.

\subsection{Tropical curves and curve classes}

If you \texttt{print(D)} for \texttt{D} a scattering diagram of some pair $(X,D)$, this will also display the classes $\beta_{\mathfrak{d}}$ of $\mathfrak{d}$ as described in \S\ref{S:classes}. These are computed using the tropical curves $h_{\mathfrak{d}}$ obtained by completing $\mathfrak{d}$, see \S\ref{S:scattrop}. You can also display $h_{\mathfrak{d}}$. To access a particular ray in a diagram \texttt{D} you can type \texttt{ray=D.rays[i]} where \texttt{i} is some index that shouldn't be larger than the number of rays in \texttt{D}. Then you can print \texttt{ray.tropical{\textunderscore}curve}. For instance,

\begin{minipage}{\textwidth}
\vspace{5mm}
\hspace{1cm}
\begin{tabular}{m{6cm} l}
\texttt{D=Diagram('P2',3).scattering(3,'P2')} & \\
\texttt{ray=D.rays[79]} & \\
\texttt{D.draw(ray.tropical{\textunderscore}curve)} &
\end{tabular}
\vspace{5mm}
\end{minipage}

\subsection{Broken lines}

The superpotential of a Landau-Ginzburg model that is mirror dual to $(X,D)$ can be computed using broken lines on the tropicalization $(B,\mathscr{P})$, see \cite{CPS}. By \cite{Gra2} broken lines correspond to $2$-marked log Gromov-Witten invariants $R_{p,q}(X,D)$ in the same way that tropical curves correspond to $1$-marked log Gromov-Witten invariants. In \cite{GRZ} broken lines were related to the open mirror map.

I will not define broken lines here, but for those who are familiar with this notion it might be interesting to note that you can compute broken lines of a diagram \texttt{D} by \texttt{D.brokenlines(pt,k)} where \texttt{pt} is some point in $\mathbb{R}^2$ and $k$ is some order. Again you can speed up the calculation by giving a case as additional parameter. Moreover, you can specify the incoming and ending exponents of your broken lines. \\
\begin{minipage}{\textwidth}
\vspace{5mm}
\hspace{5mm}
\begin{tabular}{m{6cm} l}
\texttt{D=Diagram('P2',2).scattering(2,'P2')} & \\
\texttt{D1=D.brokenlines((9,0.13),2,'P2',exp{\textunderscore}end=[(-1,0)],exp{\textunderscore}in=[(5,0)])} &
\end{tabular}
\vspace{5mm}
\end{minipage}
gives the broken lines to compute the log Gromov-Witten invariant $R_{1,5}(\mathbb{P}^2,E)=25$ of conics meeting $E$ in a fixed point with multiplicity $1$ and another point with multiplicity $5$ as in \cite{GRZ}, Appendix B.



\end{document}